\documentclass[smallextended]{svjour3}

\smartqed
\usepackage{amssymb,amsxtra}
\usepackage{dsfont,latexsym,enumerate}
\usepackage{url}
\usepackage{mathrsfs}

\usepackage{color}
\definecolor{citation}{rgb}{0.2,0.6,0.2}
\definecolor{formula}{rgb}{0.1,0.2,0.6}
\definecolor{url}{rgb}{0,0,0.4}
\usepackage[colorlinks=true,linkcolor=formula,urlcolor=url,citecolor=citation]{hyperref}

\newlength{\defbaselineskip}
\newcommand{\setlinespacing}[1]
           {\setlength{\baselineskip}{#1 \defbaselineskip}}

\usepackage{amsthm}

\numberwithin{equation}{section}
\allowdisplaybreaks[4]

\def\vs{\vspace{1mm}}
\def\dd{d_{\textrm{o}}}

\allowdisplaybreaks
\makeatletter
\DeclareRobustCommand*{\bfseries}{%
	\not@math@alphabet\bfseries\mathbf
	\fontseries\bfdefault\selectfont
	\boldmath
}

\makeatother

\def\vs{\vspace{1mm}}

\def\dd{d_{\textrm{o}}}

\allowdisplaybreaks
\makeatletter
\DeclareRobustCommand*{\bfseries}{%
	\not@math@alphabet\bfseries\mathbf
	\fontseries\bfdefault\selectfont
	\boldmath
}

\makeatother

\numberwithin{equation}{section}
\allowdisplaybreaks[4]

\definecolor{forestgreen(web)}{rgb}{0.13, 0.55, 0.13}
\definecolor{darkpowderblue}{rgb}{0.0, 0.2, 0.6}
\definecolor{indigo(dye)}{rgb}{0.0, 0.25, 0.42}
\definecolor{byzantium}{rgb}{0.44, 0.16, 0.39}
\definecolor{cornellred}{rgb}{0.7, 0.11, 0.11}

\def \R {{\mathds {R}}}
\def \e{{\varepsilon}}
\def \p{{\phi}}
\def \l{{\mathcal{L}}}
\def \r{{\mathds{R}}}
\def \h{{\mathds{H}}}

\def\Xint#1{\mathchoice
	{\XXint\displaystyle\textstyle{#1}}%
	{\XXint\textstyle\scriptstyle{#1}}%
	{\XXint\scriptstyle\scriptscriptstyle{#1}}%
	{\XXint\scriptscriptstyle\scriptscriptstyle{#1}} %
	\!\int}
\def\XXint#1#2#3{{\setbox0=\hbox{$#1{#2#3}{\int}$}
		\vcenter{\hbox{$#2#3$}}\kern-.5\wd0}}

\def\dashint{\Xint-}

\begin{document}

	\title{{\normalsize \rm to appear in\,} {\textrm {\it \large J. Geom. Anal.} {\large\bf 33} {\rm \large(2023), no.~3, \normalsize Art.~77.}} \\[3ex] H\"older continuity and boundedness estimates for nonlinear fractional equations in the Heisenberg group\thanks{
	The authors are members of~\,``Gruppo Nazionale per l'Analisi Matematica, la Probabilit\`a e le loro Applicazioni (GNAMPA)'' of Istituto Nazionale di Alta Matematica (INdAM).  
	The second author is supported by the University of Parma via the project ``Regularity, Nonlinear Potential Theory and related topics''.  \\ The authors are in debt with the referee for his/her deep understanding of  this paper, and for his/her suggestions which contributed to improve the main results. }}

   \author{Maria Manfredini \and Giampiero Palatucci \and Mirco Piccinini \and Sergio Polidoro}

\institute{M. Manfredini \at 
          Dipartimento di Scienze Fisiche, Informatiche e Matematiche\\
          Universit\`a degli Studi di Modena e Reggio Emilia\\ Via G. Campi 213/B, 41121 Modena, Italy\\
          \email{maria.manfredini@unimore.it}
          \and
          G. Palatucci \at 
          Dipartimento di Scienze Matematiche, Fisiche e Informatiche,\\
          Universit\`a di Parma\\
          Parco Area delle Scienze 53/a, Campus, 43124 Parma, Italy\\
          \email{giampiero.palatucci@unipr.it}
         \and M. Piccinini \at 
         Dipartimento di Scienze Matematiche, Fisiche e Informatiche,\\
          Universit\`a di Parma\\
          Parco Area delle Scienze 53/a, Campus, 43124 Parma, Italy\\\email{mirco.piccinini@unipr.it}
          \and S. Polidoro \at
          Dipartimento di Scienze Fisiche, Informatiche e Matematiche\\
          Universit\`a degli Studi di Modena e Reggio Emilia\\ Via G. Campi 213/B, 41121 Modena, Italy\\
          \email{sergio.polidoro@unimore.it}
          }

       \titlerunning{Nonlinear fractional equations in the Heisenberg group}
	\maketitle
	
	\begin{abstract}
		We extend the celebrate De Giorgi-Nash-Moser theory to a wide class of nonlinear equations driven by nonlocal, possibly degenerate, integro-differential operators, whose model is the fractional $p$-Laplacian operator on the Heisenberg-Weyl group $\h^n$. Amongst other results, we  prove that the weak solutions to such a class of problems are bounded and H\"older continuous, by also establishing general estimates as fractional Caccioppoli-type estimates with tail and logarithmic-type estimates.
	
	\keywords{Quasilinear nonlocal operators \and fractional Sobolev spaces \and H\"older continuity \and Heisenberg group \and fractional sublaplacian\vspace{1mm}}
 \subclass{35D10 \and 35B45 \and
		 35B05 \and 35H05 \and 35R05 \and 47G20\vspace{1mm}} 
   
	\end{abstract}

\vspace{1cm}

	\setcounter{equation}{0}\setcounter{theorem}{0}

	%
	%
	
	\vspace{2mm}
	
%
	\setlinespacing{1.02}
	
	\section{Introduction}
	The aim of the present paper is to prove some regularity results for the (weak) solutions to a wide class of nonlinear integro-differential equations.
	
	Let $\Omega$ be an open bounded subset of the Heisenberg group~$\h^n$, for $n\geq 1$, the class of problems we are dealing with are the following,
	\begin{equation}\label{problema}
		\begin{cases}
			\l u   = f & \text{in} \ \Omega,\\[0.5ex]
			u   =g & \text{in} \ \h^n \smallsetminus \Omega,
		\end{cases}
	\end{equation}
where  
	the nonlocal boundary datum $g$ belongs to the fractional space~$ W^{s,p}(\h^n)$, the datum $f \equiv f(\cdot,u) \in L^\infty_{\textrm{loc}}(\h^n)$ locally uniformly in~$\Omega$, and
	the leading operator~$\l$ is an integro-differential operator of differentiability exponent $s\in(0,1)$ and summability exponent $p>1$ given by
	\begin{equation}
		\label{operatore}
		\l u (\xi) = P.~\!V. \int_{\h^n}\frac{|u(\xi)-u(\eta)|^{p-2}\big(u(\xi)-u(\eta)\big)}{\dd(\eta^{-1}\circ \xi)^{Q+sp}}\,{\rm d}\eta, \qquad \xi \in \h^n,
	\end{equation}
	with $\dd$ being a homogeneous norm on $\h^n$, and $Q=2n+2$ the usual homogeneous dimension of $\h^n$.
	In the display above, the symbol $P.~\!V.$ stands for ``in the principal value sense''.
	We immediately refer the reader to Section~\ref{sec_preliminaries} for the precise definitions of the involved quantities and related properties, as well as further observations in order to relax some of the assumptions presented above.
	
	As a model example in the class of the problems in~\eqref{problema}, one can consider 
	the classic fractional Dirichlet problem, despite in such a case the difficulties arising from the nonlinear growth in the definition in~\eqref{operatore} actually disappear; 
	that is,
	\begin{equation}\label{pb_fractio}
		\begin{cases}
			(-\Delta_{\h^n})^s u   = 0 & \text{in} \ \Omega,\\[0.5ex]
			u   =g & \text{in} \ \h^n \smallsetminus \Omega,
		\end{cases}
	\end{equation}
	where as usual
	the symbol~$(-\Delta_{\h^n})^s$ refers to the fractional  
	sublaplacian on the Heisenberg group, defined in the suitable fractional Sobolev spaces~$H^s(\h^n)$ for any $s\in (0,1)$ as follows, 
	\begin{equation}\label{def_subf}
		(-\Delta_{\h^n})^su(\xi):=C(n,s)\lim_{\delta \rightarrow 0^+}\int_{\h^n \smallsetminus B_\delta(\xi)}\frac{u(\xi)-u(\eta)}{|\eta^{-1}\circ \xi|_{\h^n}^{Q+2s}}\,{\rm d}\eta, \qquad \xi \in \h^n.  	\end{equation}
	In the Gagliardo-type kernel above the symbol $|\cdot|_{\h^n}$ denotes the standard Heisenberg gauge.
	\vspace{2mm}
	
	In the last decades, a great attention has been focused on the study of problems involving fractional equations, both from a pure mathematical point of view and for concrete applications since they naturally arise in many different contexts. Despite its relatively short history, the literature is really too wide to attempt any comprehensive treatment in a single paper; we refer for instance to the paper~\cite{DPV12} for an elementary introduction to fractional Sobolev spaces and for a quite extensive (but still far from exhaustive) list of related references. For what concerns the regularity  and related results, the theory in the fractional Euclidean case has very recently shown many fundamental progresses. In this respect, by solely focusing on the specific goal of the present manuscript, which does basically consist in the generalization of the celebrated De Giorgi-Nash-Moser theory to the fractional Heisenberg setting, even in the {\it nonlinear} case when $p\neq2$, many important results for the Euclidean counterpart of~\eqref{problema} have been obtained, as for instance boundedness, Harnack inequalities, and H\"older continuity (up to the boundary) for the fractional $p$-Dirichlet problem in~\cite{DKP14,DKP16,KKP16,IMS16,BLS18}; for a survey on the results in the aforementioned papers and other related investigations we refer the interested reader to the paper~\cite{Pal18}.
	
	\vspace{2mm}

	For what concerns specifically  {\it  the fractional panorama in the Heisenberg group}, we first stress that one can find different definitions of the involved operator and related extremely different approaches. In the case when $p=2$, an explicit integral definition can be found in the relevant paper~\cite{RT16}, where several Hardy inequalities for the conformally invariant fractional powers of the sublaplacian are proven, also paying attention to the sharpness of the involved constants, and thus extending to the Heisenberg group some of the important results in~\cite{FLS08}, as well as extending to the fractional framework some Heisenberg type uncertainty inequalities proven for the sublaplacian in~\cite{DGP11}. We refer also to~\cite{CCR15} for related Hardy and uncertainty inequalities on general stratified Lie groups, involving fractional powers of the Laplacian, and also to~\cite{AM18}, where, amongst other important results, Sobolev and Morrey-type embeddings are derived for fractional order Sobolev spaces.  
	
	Still in the linear case when~$p=2$, very relevant results have been obtained based on the construction of fractional operators via  a Dirichlet-to-Neumann map associated to degenerate elliptic equation, as firstly seen for the Euclidean framework in~\cite{CS07}. For this, we would like to mention the very general approach in the recent series of papers~\cite{GT21,GT21b}; the Liouville-type theorem in~\cite{CT16}; the Harnack and H\"older results in Carnot groups in~\cite{FF15}; the results in the context of CR geometry in~\cite{FGMT15}; the connection with the fractional perimeters of sets in Carnot group in~\cite{FMPPS18}.

	For what concerns the more general situation as in~\eqref{operatore} when a $p$-growth exponent is considered, in our knowledge, a regularity theory is very far from be complete; nonetheless, very interesting estimates have been proven, as, e.~\!g., in~\cite{KS18,WD20,KS20}, and fundamental functional inequalities have been very recently obtained in the nonlocal framework even for more general metric spaces (see~\cite{DLV21}).

	\vspace{2mm}

	It is worth noticing that the equation~in~\eqref{problema} inherits both the difficulties arising from the noneuclidean geometrical structure and those naturally arising from the nonlocal character of the involved  integro-differential operators. More than this, 
	it is worth pointing out that the fractional operators~$\mathcal{L}$ in~\eqref{operatore} present as well the typical issues given by their {\it nonlinear} growth behavior. For this, some very important tools recently introduced in the nonlocal theory and successfully applied in the fractional sublaplacian on the Heisenberg group, as the celebrated Caffarelli-Silvestre (\!\!\cite{CS07}) $s$-harmonic extension mentioned above, and the approach via Fourier representation, 
	as well as
	other  successful tricks, like for instance the pseudo-differential commutator compactness in~\cite{PP14}, 
	the commutator estimates in~\cite{Sch16}, and many others, seem not to be adaptable to the framework considered here. However, even in the nonlinear noneuclidean framework considered here, we will be able to extend part of the strategy developed in~\cite{DKP14,DKP16} where it has been introduced a special quantity, {\it the nonlocal tail} of a fractional function, which has revealed to play a fundamental role to understand the nonlocality of the nonlinear operator~$\mathcal{L}$.
	In our settings, the nonlocal tail of a function $u \in W^{s,p}(\h^n)$ in a ball $B_R(\xi_0) \subset \h^n$ of radius $R>0$ and centered in $\xi_0\in\h^n$, is defined as follows,
	\begin{equation}
		\label{def_tail}
		\textup{Tail}(u;\xi_0,R):= \left(R^{sp} \int_{\h^n \smallsetminus B_R(\xi_0)}|u(\xi)|^{p-1}|\xi_0^{-1} \circ \xi|_{\h^n}^{-Q -sp}\,{\rm d}\xi\right)^{\frac{1}{p-1}}\!.
	\end{equation}
	In the standard Euclidean framework, the nonlocal tail has already proven to be a key-point in the proofs when a fine quantitative control of the long-range interactions, naturally arising when dealing with nonlocal operators as in~\eqref{operatore}, is needed. As mentioned before, right after its introduction, this quantity
	has been subsequently used in many recent results on nonlinear fractional equations; see for instance the subsequent results proven in~\cite{KMS15,KKP16,IMS16,KKP17,BLS18} and the references therein.
	\vspace{2mm}
	
	We are now in a position to state our main results. Here below we assume that the datum $f \equiv f(\cdot,u)$ belongs to $L^\infty_{\textrm{loc}}(\h^n)$ locally uniformly in~$\Omega$. However, as we will remark in forthcoming Section~\ref{sec_main}, such an assumption can be suitably replaced; see in particular Remark~\ref{rem_datum} there.
	Our first result describes the local boundedness of weak subsolutions.
	\begin{theorem}[{\bfseries Local boundedness}]\label{teo_bdd}
		Let $s \in (0,1)$ and $p \in (1,\infty)$, let $u \in W^{s,p}(\h^n)$ be a weak subsolution to~\textup{(\ref{problema})}, and let $B_r \equiv B_r(\xi_0)   \subset \Omega$. Then the following estimate holds true, for any $\delta \in (0,1]$,
		\begin{equation}		\label{eq_bdd}
			\sup_{B_{r/2}}u \, \leq\,   \delta \,\textup{Tail}(u_+; \xi_0, r/2) + \textbf{c}\,\delta^{-\frac{(p-1)Q}{sp^2}}  \left(\,\dashint_{B_r}u_+^p{\rm d}\xi\right )^\frac{1}{p}, 
		\end{equation}
		where $\textup{Tail}(u_+;\xi_0,r/2)$ is defined in~\eqref{def_tail}, $u_+:=\max\left\{u,\, 0\right\}$ is the positive part of the function $u$,  and the constant $\textbf{c}$ depends only on $n,p,s$ and $\|f\|_{L^\infty(B_r)}$.
	\end{theorem}
	
	We would like to stress the presence of the parameter~$\delta$ which allows an interpolation between the local and nonlocal terms in~\eqref{eq_bdd}. 
	In our knowledge, the boundedness result presented in Theorem~\ref{teo_bdd} above is new even in the linear case when~$p=2$.
	
	The second result provides the desired local H\"older continuity for the weak solutions to problem~\eqref{problema}. As expected, the nonlocal tail of the solutions naturally arises in the main estimate.
	\begin{theorem}[{\bfseries H\"older continuity}]\label{teo_holder}
		Let $s \in (0,1)$, $p \in (1,\infty)$, and let $u \in W^{s,p}(\h^n)$ be a solution to~\textup{(\ref{problema})}. Then $u$ is locally H\"older continuous in~$\Omega$. In particular, there are constants $\alpha < sp/(p-1)$ and $\textbf{c}>0$, both depending only on $n,p,s$ and $\|f\|_{L^\infty(B_r)}$, such that if $B_{2r}(\xi_0) \subset \Omega$ then
		\begin{equation}
			\mathop{\textup{osc}}\limits_{B_\rho (\xi_0)} \, u 
			\, \leq\,  \textbf{c}\, \left(\frac{\rho}{r}\right)^\alpha \left[\textup{Tail}(u;\xi_0,r) + \left(\,\dashint_{B_{2r}(\xi_0)}|u|^p \,{\rm d}\xi\right)^\frac{1}{p}\right],
		\end{equation} 
		for every $\rho\in (0,r)$.
	\end{theorem}
	The theorem above provides an extension of the classical results by De Giorgi-Nash-Moser  
	to the nonlocal framework on the Heisenberg group. In the linear case, when $p=2$, 
for what concerns classical H\"older regularity results for linear integro-differential operators in a very wide class of metric measure spaces, we refer to the important paper by Chen et Kumagai~\cite{CK08}. 
	Still in the linear case, it is also worth mentioning some related  regularity results in~\cite{FF15}, where the authors deal with linear operators related to~\eqref{operatore} 
	 by making use of the Neumann-to-Dirichlet extension, which -- as said above -- is not applicable in our nonlinear setting. 
	
	\vspace{2mm}
	
	In both the proof of the H\"older continuity result and that of the local boundedness one, a crucial role is played by the precise estimates stated in the following theorems, the Caccioppoli-type estimate (see~Theorem~\ref{teo_caccioppoli} below) and the logarithmic-type one (see forthcoming Lemma~\ref{lem_log}). We believe that these results could have their own interest in the analysis of equations involving the (nonlinear) fractional sublaplacian on the Heisenberg group, and related integro-differential operators. The first of them states a natural extension in our framework of the Caccioppoli inequality with tail, by showing that even in such a noneuclidean case one can take into account a suitable tail in order to detect deeper informations on the regularity of the solutions.
	
	\begin{theorem}[{\bfseries Caccioppoli estimates with tail}]\label{teo_caccioppoli}
		Let  $s \in (0,1)$, $p \in (1,\infty)$, and let $u \in W^{s,p}(\h^n)$ be a weak subsolution to~\eqref{problema}. Then, for any~$B_r \equiv B_r(\xi_0) \subset \Omega$ and any~nonnegative $\p \in C^\infty_0(B_r)$, the following estimate holds true
		
		\begin{align}\label{caccioppoli}
			\int_{B_r} &\int_{B_r} |\eta^{-1} \circ \xi|_{\h^n}^{-Q-sp}  |w_+(\xi)\p(\xi)-w_+(\eta)\p(\eta)|^p \, \,{\rm d}\xi \,{\rm d}\eta\notag\\*
			& \leq \textbf{c}\int_{B_r}\int_{B_r} |\eta^{-1} \circ \xi|_{\h^n}^{-Q-sp} w_+^p(\xi)|\p(\xi)-\p(\eta)|^p \, \,{\rm d}\xi \,{\rm d}\eta\\*
			&\quad+ \textbf{c} \int_{B_r}w_+(\xi)\p^p(\xi) \, \,{\rm d}\xi \biggl(\sup_{\eta \in \textup{supp}\, \p}\int_{\h^n \smallsetminus B_r} |\eta^{-1} \circ \xi|_{\h^n}^{-Q-sp} w_+^{p-1}(\xi) \, \,{\rm d}\xi\notag   \\*
			& \hspace{4.5cm} +\|f\|_{L^\infty(B_r)}\biggr)\notag
		\end{align}
		where $w_+ := (u-k)_+$ with $k\in\mathds{R}$, and $\textbf{c}=\textbf{c}\,(
		n,p,s)$.
	\end{theorem}

	\begin{lemma}[{\bfseries Logarithmic Lemma}]
		\label{lem_log}
		Let $s \in (0,1)$, $p \in (1,\infty)$, and let $u \in W^{s,p}(\h^n)$ be a weak solution to~\textup{(\ref{problema})} such that $u \geq 0$ in $B_{R} \equiv B_R(\xi_0) \subset \Omega$. Then, there exists a constant $\tilde{c}\in [1,+\infty)$ such that the following estimate holds for any~$B_r \equiv B_r(\xi_0) \subset B_{\frac{R}{2\tilde{c}}}(\xi_0)$ and any $d>0$,
		\begin{eqnarray}\label{loga}
			&& \int_{B_r}\int_{B_r}|\eta^{-1}\circ \xi|_{\h^n}^{-Q-sp} \, \left|\log \left(\frac{u(\xi)+d}{u(\eta)+d}\right)\right|^p \,{\rm d}\xi \,{\rm d}\eta \nonumber\\* &&\qquad+\int_{B_r}\big(f(\xi,u)\big)_+\big(u(\xi)+d\big)^{1-p}\, \,{\rm d}\xi\notag\\*[0.5ex]     
			&&\qquad \leq \textbf{c}r^{Q-sp} + \textbf{c}d^{1-p}\frac{r^Q}{R^{sp}}\left\{\left[\textup{Tail}(u_-;\xi_0,R)\right]^{p-1}+1\right\}\\*
			&& \qquad\quad +\, \textbf{c}\|f\|_{L^\infty(B_{r})}\int_{B_{2r}}(u(\xi)+d)^{1-p} \, \,{\rm d}\xi.\notag
		\end{eqnarray}
		where $\textup{Tail}(u_-;\xi_0,R)$ is defined in~\eqref{def_tail}, $u_-:=\max\{-u,0\}$ and $\textbf{c}$ depends only on $n,p$ and $s$.
	\end{lemma} 
	\vspace{3mm}
	
	Starting from the results proven in the present paper, several questions naturally arise:
	\\*[0.3ex]
	\indent $\bullet$  Firstly, it is worth remarking that here we treat general weak solutions, namely by truncation and dealing with the resulting error term as a right hand-side, in the same flavour of the papers~\cite{DKP14,DKP16}, in the spirit of De~Giorgi-Nash-Moser. However, one could approach the same family of problems by focusing solely to bounded viscosity solutions in the spirit of Krylov-Safonov, as in the important paper~\cite{Sil06}. 
	\vspace{1mm}
	
	$\bullet$ Consequently, a second natural question is whether or not, and under which assumptions on the structural quantities, the viscosity solutions to nonlocal equations in the Heisenberg group are indeed fractional harmonic functions and/or weak solutions, and vice versa. In this respect, let us observe that one cannot plainly apply the results for  $p$-fractional minimizers as obtained in the recent paper~\cite{KKP17} together with those in~\cite{KKL19}, whose proofs seem to be feasible only for a restrict class of kernels which cannot include modulating coefficients or other variations.  
	\vspace{1mm}
	
	$\bullet$ Third, in the same spirit of the series of paper by Brasco, Lindgren, and Schikorra~\cite{BL17,BLS18}, one would expect higher differentiability and other additional regularity results for the bounded solutions to nonlocal equations in the Heisenberg group. 
	It could be useful to start from  the estimates  in the aforementioned papers obtained for the standard fractional $p$-Laplace equation.
	\vspace{1mm}
	
	$\bullet$ Also, again in clear accordance with the Euclidean counterpart, one would expect self-improving properties of the solutions to~\eqref{problema}. For this, one should extend the recent nonlocal Gehring-type theorems proven in~\cite{KMS15,Sch16}. 
	\vspace{1mm}
	
	$\bullet$ Also, one could expect H\"older continuity and other regularity results for the solutions to a strictly related class of problems; that is, by adding in~\eqref{problema} a second integral-differential operators, of differentiability exponent~$t>s$ and summability growth~$q>1$, controlled by the zero set of a modulating coefficient: the so-called nonlocal double phase problem, in the same spirit of the Euclidean case treated in~\cite{DFP19,ZTR21}, starting from the pioneering results in the local case, when $s=1$, by Colombo, Mingione and many others; see for instance~\cite{DM20,DM20b} and the references therein.
	\vspace{1mm} 
	
	$\bullet$ Also, mean value properties for solutions to general nonlinear fractional operators, and their stability in the linear case when~$p=2$, could lead to very tricky situations in fractional non-Euclidean framework; we refer to the very recent papers~\cite{BDV20,BS21} and the references therein.
	\vspace{1mm}
		
	$\bullet$ Moreover, to our knowledge, nothing is known about the regularity  for  solutions to  parabolic nonlocal integro-differential equations involving the nonlinear operators in~\eqref{operatore}.
	\vspace{1mm}
	
	$\bullet$ Finally, by starting from the estimates proven in the present paper, in~\cite{PP21} weak and strong 
	Harnack inequalities for the solutions to~\eqref{problema} are proven. As expected, a tail contribution naturally appears in those estimates in order to control the nonlocal contributions coming from far. We refer also to~\cite{Pic22} for regularity results (up to the boundary) for very general boundary data, and for the related obstacle problems.

	\vspace{3mm}
	
	{\it To summarize}.\,~The result in the present paper seems to be one of the first 
	 concerning regularity properties of nonlinear nonlocal equations in the Heisenberg group. We prove that one can extend to the Heisenberg setting the strategy successfully applied in the fractional Euclidean case~(\!\!~\cite{DKP14,DKP16,IMS16,KMS15}); from another point of view our results can be seen as the (nonlinear) nonlocal extension of the Heisenberg counterpart of the celebrated De Giorgi-Nash-Moser theory (\!\!~\cite{MM07,MZ21}). Moreover, since we derive all our results for a general class of nonlinear integro-differential operators, via our approach by taking into account all the nonlocal tail contributions in a precise way, we obtain alternative proofs that are new even in the by-now classical case of the pure fractional sublaplacian operator $(-\Delta_{\h^n})^s$.
	Finally, we prove a boundedness estimate allowing an interpolation between the local and nonlocal contributions, which seems to be new even in the linear case.
	We believe our estimates to be important in a forthcoming nonlinear nonlocal theory in the Heisenberg group.

	\vspace{3mm}
	{\it The paper is organized as follows}.\,~In Section~\ref{sec_preliminaries} below we set up notation and terminology, and we briefly recall our underlying geometrical structure, by also recalling the involved functional spaces, and providing a few remarks on the assumptions on the data. The whole Section~\ref{sec_caccioppoli} and Section~\ref{sec_log} are devoted to the proof of the Caccioppoli inequality with tail, and the Logarithmic Lemma, respectively.  
	In the last two sections we are finally able to prove the boundedness result in~Theorem~\ref{teo_bdd}, and the H\"older continuity of the weak solutions~$u$ to~\eqref{problema}.
	
	%
	%
	\vspace{2mm}
	\section{Preliminaries}\label{sec_preliminaries}
	It is convenient to fix some notation which will be used throughout the rest of the paper. Firstly, notice that we will follow the usual convention of denoting by $\textbf{c}$ a general positive constant which will not necessarily be the same at different occurrences and which can also change from line to line. For the sake of readability, dependencies of the constants will  be often omitted within the chains of estimates, therefore stated after the estimate.

	\subsection{The Heisenberg-Weyl group}
	We start by introducing some definitions and briefly setting up the notation concerning the Heisenberg group. For further details we refer to the book by Bonfiglioli, Lanconelli and Uguzzoni, \cite{BLU07}.
	
	\vspace{2mm}
	As customary, we identify the Heisenberg group~$\h^n$ with $\r^{2n+1}$. Points in $\h^n$ are denoted by
	$$
	\xi := (z,t) = (x_1,\dots,x_n,y_1,\dots,y_n,t).
	$$
	The related group multiplication is given by
	\begin{eqnarray*}
		\xi \circ \xi' &:= &\Big(x+x',\, y+y',\, t+t'+2\langle y,x'\rangle-2\langle x,y'\rangle \Big) \\*[0.3ex]
		& = & \left( x_1+x_1', ..., x_n+x_n',\, y_1+y_1', ..., y_n+y_n',\, t+t' +2 \sum_{i=1}^n\big(
		y_ix_i'-x_iy'_i\big) 
		\right).
	\end{eqnarray*} 
	
	One can check that $(\r^{2n+1}, \circ)$ is a Lie group with identity element the origin~$0$ and inverse $\xi^{-1}=-\xi$. Moreover, one can consider the following automorphism group ${\Phi}_\lambda$ on $\r^{2n+1}$,
	$$
	\begin{aligned}
		{\Phi}_\lambda :  \r^{2n+1} &\longrightarrow \r^{2n+1}\\
		\xi &\longmapsto {\Phi}_\lambda(\xi):=\big(\lambda x, \, \lambda y, \,\lambda^2 t\big),
	\end{aligned}
	$$
	so that the group $\h^n=\big(\r^{2n+1},\,\circ,\, {\Phi}_\lambda\big)$ is a homogeneous Lie group; that is, the so-called {\it Heisenberg-Weyl group in} \ $\r^{2n+1}$. 
	\vspace{2mm}
	
	The Jacobian  
	basis of the Heisenberg Lie algebra $\mathbf{h}^n$ of $\h^n$ is given by
	$$
	X_j := \partial_{x_j} +2y_j \partial_t, \quad
	X_{n+j}:= \partial_{y_j}-2x_j \partial_t, \quad 1 \leq j \leq n, \quad
	T= \partial_t.
	$$
	Since 
	$$
	[X_j,X_{n+j}]=-4\partial_t \quad \text{for every} \ 1 \leq j \leq n,
	$$
	it follows
	\begin{eqnarray*}
		&& \textup{rank}\Big(\textup{Lie}\{X_1,\dots,X_{2n},T\}(0,0)\Big) \\*
		&& \qquad \qquad  = \  \textup{span}\big\{\partial_{x_1},\dots,\partial_{x_n},\partial_{y_1},\dots,\partial_{y_n},-4\partial_t\big\} 
		\ = \ 2n+1,
	\end{eqnarray*}
	which is the Euclidean dimension of $\h^n$, whereas we denote by $Q$ its homogeneous dimension
	$$
	Q=2n+2.
	$$
	This shows that $\h^n$ is a Carnot group with the following stratification
	$$
	\mathbf{h}^n = \textup{span}\{X_1,\dots,X_{2n}\}\oplus \textup{span}\{T\}.
	$$
	\vspace{2mm}
	
	Moreover, let $\Omega \subset \h^n$ be a domain. For $u\in C^1(\Omega;\,\r)$ we define the subgradient $\nabla_{\h^n} u$ by
	$$
	\nabla_{\h^n} u (\xi):= \Big(X_1u(\xi),\dots, X_{2n}u(\xi)\Big),
	$$
	and
	$$
	|\nabla_{\h^n}u|^2 := \sum_{j=1}^{2n}|X_ju|^2.
	$$ 
	
	\begin{defn}
		\label{def_homnorm}
		A \textup{homogeneous norm} on $\h^n$ is a continuous function {\rm (}with respect to the Euclidean topology\,{\rm )} ${d_{\rm o}} : \h^n \rightarrow [0,+\infty)$ such that:
		\begin{enumerate}[\rm(i)]
			\item{
				${d_{\rm o}}({\Phi}_\lambda(\xi))=\lambda {d_{\rm o}}(\xi)$, for every $\lambda>0$ and every $\xi \in \h^n$;
			}\vspace{1mm}
			\item{
				${d_{\rm o}}(\xi)=0$ if and only if $\xi=0$.
			}
		\end{enumerate}
		Moreover, we say that the homogeneous norm~${d_{\rm o}}$ is {\rm symmetric} if 
		$$
		{d_{\rm o}}(\xi^{-1})={d_{\rm o}}(\xi), \qquad \forall\xi \in \h^n.
		$$
	 \end{defn}
	\begin{rem}
		Let ${d_{\rm o}}$ be a homogeneous norm on $\h^n$. Then the function~$\Psi$ defined on the set of all pairs of elements of~$\h^n$ by
		$$
		\Psi(\xi,\eta):={d_{\rm o}} (\eta^{-1}\circ \xi)
		$$
		is a pseudometric on $\h^n$.
	\end{rem}
	Consider now the standard homogeneous norm on $\h^n$,
	\begin{equation}
		\label{work norm}
		|\xi|_{\h^n}= \left(|z|^4 +t^2\right)^\frac{1}{4}, \qquad \forall \xi=(z,t) \in \h^n.
	\end{equation}
	For any fixed $\xi_0 \in \h^n$ and $R>0$, the ball~$B_R(\xi_0)$ with center $\xi_0$ and radius $R$ is given by
	$$
	B_R(\xi_0):=\Big\{\xi \in \h^n : |\xi_0^{-1}\circ \xi|_{\h^n} < R\Big\}.
	$$
	\vspace{2mm}
	
	We conclude this section with some properties of the homogeneous norm on $\h^n$ that will be useful in the rest of the paper.
	\begin{prop}[{\bfseries Equivalence of the homogeneous norm}]
		\label{prop1}
		Let ${d_{\rm o}}$ be a homogeneous norm on $\h^n$. Then there exists a constant $\textbf{c}>0$ such that
		$$
		\textbf{c}^{-1}|\xi|_{\h^n}\leq {d_{\rm o}}(\xi) \leq \textbf{c}|\xi|_{\h^n}, \qquad \forall \xi \in \h^n.
		$$
	\end{prop}
	In view of the preceding proposition, in most of the forthcoming proofs one can simply take into account the pure homogeneous norm defined in~\textup{(\ref{work norm})} with no modifications at all.
	 
	\vspace{1mm} 
	\begin{prop}[{\bfseries Pseudo-triangle inequalities}]
		\label{prop2}
		Let ${d_{\rm o}}$ be a homogeneous norm on $\h^n$. Then there exists a constant $\tilde{c}>0$ such that the following statements are satisfied:
		\begin{enumerate}
			\item[\rm (1)]${d_{\rm o}}(\xi \circ \eta)\leq \tilde{c}({d_{\rm o}}(\xi)+{d_{\rm o}}(\eta))$;
			\item[\rm (2)] ${d_{\rm o}}(\xi \circ \eta)\geq \frac{1}{\tilde{c}}{d_{\rm o}}(\xi)-{d_{\rm o}}(\eta^{-1})$;
			\item[\rm (3)] ${d_{\rm o}}(\xi \circ \eta) \geq \frac{1}{\tilde{c}}{d_{\rm o}}(\xi)-\tilde{c}{d_{\rm o}}(\eta)$.
		\end{enumerate}
	\end{prop}
	For a proof of the previous propositions we refer to~Proposition~5.1.4 and Proposition~5.1.7 in~\cite{BLU07}.  
		\begin{rem}\label{constant_c}
		In the case when the homogeneous norm~${d_{\rm o}}$ does reduce to the standard norm~$|\cdot|_{\h^n}$ defined	in~\eqref{work norm}, 
	 the constant~$\tilde{c}$ given by Proposition~\textup{\ref{prop2}} can be chosen equal to 1. For the proof, we refer 
 to{\rm~\cite{Cyg81}};	
	see also Example~{\rm 5.1} in{\rm~\cite{BFS17}}.
		However, in view of possible generalizations of the results obtained in the present paper, as, e.~\!g., in to a more general framework involving abstract Carnot groups~$\mathds{G}$ with some homogenous norm $|\cdot|_{\mathds{G}}$ satisfying only a pseudo-triangle inequality, we would prefer to keep the constant~$\tilde{c}$ throughout all the forthcoming proofs. 
		\end{rem}
	\vspace{2mm}
	
	The computation in the result below will be used several times in the proofs in the following.
	\begin{lemma}\label{lem1}
		Let $\gamma >0$ and let~$|\cdot|_{\h^n}$ be the homogeneous norm on $\h^n$ defined in{\rm~\eqref{work norm}}. Then, 
	\begin{equation*}
		\int_{\h^n \smallsetminus B_r(\xi_0)} \frac{{\rm d}\xi}{|\xi_0^{-1}\circ \xi|_{\h^n}^{Q+\gamma}} \leq c(n,\gamma) r^{-\gamma}.
	\end{equation*}
    \end{lemma}

    \begin{proof}
	The proof is straightforward. For any $j \in \mathds{N}$ let us indicate with $B^j$ the following set
    \begin{equation*}
	B^j := \left\{ \xi \in \h^n \smallsetminus B_r(\xi_0): 2^j r \leq |\xi_0^{-1}\circ \xi|_{\h^n} \leq 2^{j+1}r\right\}.
    \end{equation*}
    Then, we have that
    \begin{align}
	\int_{\h^n \smallsetminus B_r(\xi_0)} \frac{{\rm d}\xi}{|\xi_0^{-1}\circ \xi|_{\h^n}^{Q+\gamma}}  & = \sum_{j=0}^{\infty} \,  \int_{B^j} \frac{{\rm d}\xi}{|\xi_0^{-1}\circ \xi|_{\h^n}^{Q+\gamma}}\notag\\
	& \leq \sum_{j=0}^\infty (2^jr)^{-Q-\gamma} |B_{2^{j+1}r}(\xi_0)|\notag\\
	& = c(n) r^{-\gamma}\sum_{j=0}^\infty \left(\frac{1}{2^\gamma}\right)^j \leq c(n,\gamma) r^{-\gamma}.
    \end{align}
    \end{proof}
	 
	\subsection{The setting of the main problem}\label{sec_main}
	
	We firstly need to recall some definitions and a few basic results about our fractional functional setting. For further details, we refer the reader to~\cite{AM18,KS18}.
	\vspace{1mm}
	
	Let $p \geq 1$ and $s \in (0,1)$, and let $u : \h^n \rightarrow \r$ be a measurable function; we define the Gagliardo (semi)norm of~$u$ as follows, 
	\begin{equation}
		[u]_{W^{s,p}} = \left(\int_{\h^n}\int_{\h^n}\frac{|u(\xi)-u(\eta)|^p}{|\eta^{-1}\circ \xi|_{\h^n}^{Q+sp}}\,{\rm d}\xi \,{\rm d}\eta\right)^{\frac{1}{p}}.
	\end{equation} 
	The fractional Sobolev spaces~$W^{s,p}$ on the Heisenberg group is defined as
	\begin{equation}
		W^{s,p}(\h^n):=\Big\{u \in L^p(\h^n):  [u]_{W^{s,p}} < +\infty\Big\},
	\end{equation}
	endowed with the natural fractional norm
	\begin{equation}
		\|u\|_{W^{s,p}(\h^n)}:= \Big(\|u\|_{L^p(\h^n)}^p+[u]_{W^{s,p}}^p\Big)^\frac{1}{p}, \qquad u \in W^{s,p}(\h^n).
	\end{equation}
	Similarly, given a domain $\Omega \subset \h^n$, one can define the  fractional Sobolev space $W^{s,p}(\Omega)$ in the natural way, as follows
	\begin{equation}
		W^{s,p}(\Omega):=\left\{u \in L^p(\Omega): \,\left(\int_{\Omega}\int_{\Omega}\frac{|u(\xi)-u(\eta)|^p}{|\eta^{-1}\circ \xi|_{\h^n}^{Q+sp}}\,{\rm d}\xi \,{\rm d}\eta\right)^{\frac{1}{p}}< +\infty\right\}
	\end{equation}
	endowed with the norm
	\begin{equation}
		\|u\|_{W^{s,p}(\Omega)}:= \left(\|u\|_{L^p(\Omega)}^p+\int_{\Omega}\int_{\Omega}\frac{|u(\xi)-u(\eta)|^p}{|\eta^{-1}\circ \xi|_{\h^n}^{Q+sp}}\,{\rm d}\xi \,{\rm d}\eta\right)^\frac{1}{p}\,.
	\end{equation}
	By $W^{s,p}_0(\Omega)$ we denote the closure of $C_0^\infty(\Omega)$ in $W^{s,p}(\h^n)$. Conversely, if $v \in W^{s,p}(\Omega')$ with  $\Omega \Subset \Omega'$ and $v=0$ outside of $\Omega$ almost everywhere, then $v$ has a representative in $W_0^{s,p}(\Omega)$ as well. 
	\vspace{2mm}
	
	As expected, one can prove a fractional Sobolev embedding on the Heisenberg group. We have the following
	\begin{theorem}
		\label{sobolev}
		Let $p>1$ and $s \in (0,1)$ such that $sp<Q$. For any measurable compactly supported function $u : \h^n \rightarrow \r$ there exists a positive constant $\textbf{c}=\textbf{c}(n,p,s)$ such that
		\begin{equation*}
			\|u\|^p_{L^{p^*}(\h^n)}\, \leq \, \textbf{c} \,[u]^p_{W^{s,p}(\h^n)}\,,
		\end{equation*}
		where $p^* = {Qp}/{(Q-sp)}$ is the critical Sobolev exponent.
	\end{theorem}
	For the proof we refer to Theorem~2.5 in~\cite{KS18}, where the authors extend the strategy in the standard Euclidean settings as seen in~\cite{PSV13,DPV12}.
	
	\vspace{2mm}

	As in the classical case with $s$ being an integer, the space $W^{s,p}$ is continuously embedded in~$W^{s_1,p}$ when $s_1\leq s$, as the result below points out.
	\begin{prop}\label{sobolev immersion}
		Let $p >1$ and $0< s_1 \leq s < 1$. Let $\Omega$ be an open subset of~$\h^n$, and let $u\in W^{s,p}(\Omega)$. Then
		\begin{equation*}
			\|u\|_{W^{s_1,p}(\Omega)} \leq c \|u\|_{W^{s,p}(\Omega)},
		\end{equation*}
		for some suitable positive constant $c$ depending only on~$n, p$ and $s_1$. 	
	\end{prop}
	\begin{proof}
		We extend the strategy in the proof in the fractional Euclidean framework; see~\cite[Proposition~2.1]{DPV12}.
		
		Firstly, we can control the size of the nonlocal tail of~$u$ by its $L^p$-norm. We have
		\begin{eqnarray*}
			\int_{\Omega}\int_{\Omega \cap \{|\eta^{-1} \circ \, \xi|_{\h^n} \geq 1\}}\frac{|u(\xi)|^p}{|\eta^{-1} \circ \, \xi|_{\h^n}^{Q+s_1p}} \, {\rm d}\xi \,{\rm d}\eta & \leq& \int_{\Omega}\left(\int_{\h^n \smallsetminus B_1(0)}\frac{1}{|\tilde{\xi}|_{\h^n}^{Q+s_1p}}d\tilde{\xi}\right)|u(\xi)|^p \,{\rm d}\xi\\*[0.5ex]
			& \leq & c(n, p, s_1)\|u\|^p_{L^p(\Omega)},
		\end{eqnarray*}
		and thus
		\begin{eqnarray}\label{sec2 d1}
			&&\int_{\Omega}\int_{\Omega \cap \{|\eta^{-1} \circ \, \xi|_{\h^n} \geq 1\}}\frac{|u(\xi)-u(\eta)|^p}{|\eta^{-1} \circ \, \xi|_{\h^n}^{Q+s_1p}} \, \,{\rm d}\xi \,{\rm d}\eta \notag\\*[0.5ex]
			&& \qquad \qquad \qquad \qquad \leq 2^{p-1}\int_{\Omega}\int_{\Omega \cap \{|\eta^{-1} \circ \, \xi|_{\h^n} \geq 1\}}\frac{|u(\xi)|^p+|u(\eta)|^p}{|\eta^{-1} \circ \, \xi|_{\h^n}^{Q+s_1p}} \, \,{\rm d}\xi \,{\rm d}\eta\notag\\*[0.5ex]
			&& \qquad\qquad \qquad \qquad \leq c \|u\|^p_{L^p(\Omega)},
		\end{eqnarray}
		up to relabelling the constant~$c$.
		
		On the other hand,
		\begin{eqnarray}\label{sec2 d2}
			&&	\int_{\Omega}\int_{\Omega \cap \{|\eta^{-1} \circ \, \xi|_{\h^n} < 1\}}\frac{|u(\xi)-u(\eta)|^p}{|\eta^{-1} \circ \, \xi|_{\h^n}^{Q+s_1p}} \, \,{\rm d}\xi \,{\rm d}\eta \notag \\*
			&&\qquad\qquad	\qquad \leq \, \int_{\Omega}\int_{\Omega \cap \{|\eta^{-1} \circ \, \xi|_{\h^n} < 1\}}\frac{|u(\xi)-u(\eta)|^p}{|\eta^{-1} \circ \, \xi|_{\h^n}^{Q+sp}} \, \,{\rm d}\xi \,{\rm d}\eta.
		\end{eqnarray}
		Combining \eqref{sec2 d1} with  \eqref{sec2 d2}, we finally get
		\begin{equation*}
			\int_{\Omega}\int_{\Omega}\frac{|u(\xi)-u(\eta)|^p}{|\eta^{-1} \circ \, \xi|_{\h^n}^{Q+s_1p}} \, \,{\rm d}\xi \,{\rm d}\eta \, \leq \, c\|u\|_{L^p(\Omega)}^p + [u]^p_{W^{s,p}},
		\end{equation*}
		which yields
		$$
		\|u\|^p_{W^{s_1,p}(\Omega)}\, \leq\, (c+1)\|u\|^p_{L^p(\Omega)}+ [u]^p_{W^{s,p}} \,\leq \, c\|u\|^p_{W^{s,p}(\Omega)},
		$$
		again up to relabelling the constant~$c$.
	\end{proof}
	\vspace{2mm}
	
	We conclude this section by providing the definition of weak solution to the class of fractional problem we deal with.
	
	Let $\Omega$ be a bounded open set in $\h^n$ and $g \in W^{s,p}(\h^n)$, we are interested in the weak solutions to the following integro-differential problems,
	\begin{equation}\label{problema2}
		\begin{cases}
			\l u   = f & \text{in} \ \Omega,\\[0.5ex]
			u   =g & \text{in} \ \h^n \smallsetminus \Omega,
		\end{cases}
	\end{equation}
	where   
	the datum $f \equiv f(\cdot, u) \in L^\infty_{\textrm{loc}}(\h^n)$ locally uniformly in~$\Omega$, and
	the leading operator~$\l$ is an integro-differential operator of differentiability exponent $s\in(0,1)$ and summability exponent $p>1$ given by
	\begin{equation*} 
		\l u (\xi) = P.~\!V. \int_{\h^n}\frac{|u(\xi)-u(\eta)|^{p-2}\big(u(\xi)-u(\eta)\big)}{\dd(\eta^{-1}\circ \xi)^{Q+sp}}\,{\rm d}\eta, \qquad \xi \in \h^n,
	\end{equation*}
	with $\dd$ being a homogeneous norm on $\h^n$ in accordance with Definition~\ref{def_homnorm}.
	
	We now need to introduce some further notation. For any~$g \in W^{s,p}(\h^n)$ the classes~$\mathcal{K}^{\pm}_g (\Omega)$ of suitable fractional functions are defined by
	$$
	\mathcal{K}^{\pm}_g (\Omega):=\Big\{  v \in W^{s,p}(\h^n): (g-v)_\pm \in W^{s,p}_0(\Omega)\Big\},
	$$
	and 
	$$
	\mathcal{K}_g(\Omega) := \mathcal{K}^+_g(\Omega) \cap \mathcal{K}^-_g(\Omega) = \Big\{  v \in W^{s,p}(\h^n): v-g \in W^{s,p}_0(\Omega)\Big\}.
	$$
	We have the following
	\begin{defn}\label{solution to inhomo pbm}
		A function $u \in \mathcal{K}^-_g(\Omega)$  {\rm (}$\mathcal{K}^+_g(\Omega)$, respectively{\rm )} is a \textup{weak subsolution} {\rm (}\textup{supersolution}, resp.{\rm)} to~\textup{(\ref{problema2})} if  
		\begin{eqnarray*}
			&& \int_{\h^n}\int_{\h^n}\frac{\big|u(\xi)-u(\eta)\big|^{p-2}\big(u(\xi)-u(\eta)\big)\big(\psi(\xi)-\psi(\eta)\big)}{\dd(\eta^{-1}\circ \xi)^{Q+sp}}  \,{\rm d}\xi \,{\rm d}\eta \\*[0.5ex]
			&& \hspace{6cm} \leq \big(\geq,\textrm{resp.}\big) \int_{\h^n}f(\xi,u(\xi))\psi(\xi) \, \,{\rm d}\xi,
		\end{eqnarray*}
		for any nonnegative  $ \psi \in W_0^{s,p}(\Omega)$.
		\\	A function u is a \textup{weak solution} to~\textup{(\ref{problema2})} if it is both a weak sub- and supersolution. In particular, $u$ belongs to $\mathcal{K}_g(\Omega)$ and it satisfies 
		\begin{equation*}
			\int_{\h^n}\int_{\h^n}\frac{|u(\xi)-u(\eta)|^{p-2}(u(\xi)-u(\eta))(	\psi(\xi)-\psi(\eta))}{\dd(\eta^{-1}\circ \xi)^{Q+sp}} \, {\rm d}\xi \,{\rm d}\eta = \int_{\h^n}f(\xi,u(\xi))\psi(\xi)  \,{\rm d}\xi,  
		\end{equation*}
		for any $\psi \in W^{s,p}_0(\Omega)$.
	\end{defn}
	
	For any $u \in W^{s,p}(\h^n)$ and for any $B_R(\xi_0) \subset \h^n$ we will define the \textit{nonlocal tail of a function u in the ball} $B_R(\xi_0)$ the quantity
	\begin{equation}
		\label{tail}
		\textup{Tail}(u;\xi_0,R):= \left(R^{sp} \int_{\h^n \smallsetminus B_R(\xi_0)}|u(\xi)|^{p-1}|\xi_0^{-1} \circ \xi|_{\h^n}^{-Q -sp}\,{\rm d}\xi\right)^{\frac{1}{p-1}}.
	\end{equation}
	
	A few observations are in order.

	Firstly, we notice that, by H\"older's Inequality, since $u\in L^p(\h^n)$ and $R>0$, we have that $\textup{Tail}(u;\xi_0,R) < +\infty$.
	\vspace{2mm}
	
	Second, we have the following
	\begin{rem}\label{rem_tailspace}
		The requirement on the boundary datum~$g$ to be in the whole~$W^{s,p}(\h^n)$ can be weakened by assuming only a local fractional differentiability, namely $g \in W_{\textrm{\rm loc}}^{s,p}(\Omega)$, in addition to the boundedness of its nonlocal tail; i.~\!e., $\text{\rm Tail}(g; \xi_0, R)~<~\infty$, for some $\xi_0\in\h^n$ and some $R>0$.  This is not restrictive, and it does not bring relevant modifications in the rest of the paper. For further details on the related~``Tail space'', we refer the interested reader to papers{\rm ~\cite{KKP16,KKP17}}.
	\end{rem}
	
	\vspace{2mm}
	
	Finally, an important observation about the assumptions on the datum~$f$ in the right-hand side of~\eqref{problema2}.
	\begin{rem}\label{rem_datum}
		The presence of the datum~$f$ is a novelty with respect to the Euclidean counterpart studied in\,{\rm~\cite{DKP16}} where the authors assume the right-hand side in~\eqref{problema2} to be zero. However, as we basically will prove, the techniques there can be applied also to more general framework.
		
		In addition, in accordance with the classical elliptic theory, 
		with no important modifications in the forthcoming proofs, one could consider the case when the local boundedness assumption on the datum~$f$ is replaced by a uniformly growth control from above, as, e.~\!g., 
		$$
		|f(\xi,u)| \leq a + b|u|^{q} \quad \text{for almost everywhere} \ \xi \in \Omega \ \text{and any} \ u \in \R,
		$$
		for some suitable choice of the exponent~$q=q(n,p,s)$.
	\end{rem}

	\vspace{2mm}
	Before going into the proofs, it is worth pointing that in the rest of the paper we will only consider the case when the structural parameters $n$, $s$ and $p$ are such that~$sp\leq Q$. This is not a restriction, since, in the remaining case when $sp>Q$, the desired boundedness and H\"older continuity results are assured by the fractional Morrey embedding in the Heisenberg group; for the proof we refer for instance to~Theorem~1.5 in~\cite{AM18}.

	%
	%
	
	\vspace{2mm}
	\section{Proof of the Caccioppoli inequality with tail}\label{sec_caccioppoli}
	The aim of this section is to give a full proof of Theorem~\ref{teo_caccioppoli}. We would stress that, as in the classical Euclidean case,  in both the entire framework and  the fractional one, the {\it Caccioppoli estimates with tail}~\eqref{caccioppoli} encode all the needed information to derive the desired H\"older continuity from the minimum properties of the solutions, and it is an independent result which could be very useful in order to detect further regularity properties of the solutions to general fractional problems.
	
	\begin{proof}[Proof of Theorem~\ref{teo_caccioppoli}]\,
		Let $u$ be a weak subsolution. We firstly choose as a test function in~\ref{solution to inhomo pbm} the function
		$$
		\psi:= w_+\p^p \equiv (u-k)_+\p^p, \qquad \text{for} \ k \in \R,
		$$
		where $\p$ is any nonnegative function in~$C^\infty_0(B_r)$.
		We get
		\begin{eqnarray}\label{cacc e1}
			0 & \geq& \int_{B_r}\int_{B_r}|\eta^{-1} \circ \xi|_{\h^n}^{-Q-sp}|u(\xi)-u(\eta)|^{p-2}\notag\\*
			&& \qquad\quad \ \times \, \big(u(\xi)-u(\eta)\big)\big(w_+(\xi)\p^p(\xi)-w_+(\eta)\p^p(\eta)\big) \,{\rm d}\xi {\rm d} \eta \notag\\*
			&&   +2\int_{\h^n \smallsetminus B_r}\int_{B_r}|\eta^{-1} \circ \xi|_{\h^n}^{-Q-sp}|u(\xi)+u(\eta)|^{p-2}\notag\\*
			&&\qquad\quad \times \, \big(u(\xi)-u(\eta)\big)w_+(\xi)\p^p(\xi)  \,{\rm d}\xi {\rm d}\eta\nonumber\\*
           &&- \int_{B_r}f(\xi,u(\xi))w_+(\xi)\p^p(\xi)\, {\rm d}\xi.
		\end{eqnarray}
		Note that $\psi$ is an admissible test function since truncations of functions in~$W^{s,p}(\h^n)$ still belongs to~$W^{s,p}(\h^n)$.
		
		Let us begin by estimating the first integral on the right-hand side in~\eqref{cacc e1}. Without loss of generality, we assume that $u(\xi) \geq u(\eta)$; otherwise it just suffices to interchange the roles of $\xi$ and $\eta$ below. 
		We have
		\begin{eqnarray*}
			&& \big|u(\xi)-u(\eta)\big|^{p-2}\big(u(\xi)-u(\eta)\big)\big(w_+(\xi)\p^p(\xi)-w_+(\eta)\p^p(\eta)\big)\\*[0.5ex]
			&& \qquad  = \big(u(\xi)-u(\eta)\big)^{p-1}((u(\xi)-k)_+\p^p(\xi)-\big(u(\eta)-k)_+\p^p(\eta)\big)\\*[0.5ex]
			&&  \qquad= 
			\begin{cases}
				\big(w_+(\xi)-w_+(\eta)\big)^{p-1}(w_+(\xi)\p^p(\xi)-w_+(\eta)\p^p(\eta)) \qquad \text{for} \ u(\xi),\, u(\eta) >k,\\*
				\big(u(\xi)-u(\eta)\big)^{p-1}w_+(\xi)\p^p(\xi) \hspace{4cm}  \text{for} \ u(\xi)>k, \ u(\eta) \leq k,\\*
				0 \hspace{8.5cm} \text{otherwise}
			\end{cases}\\*[0.5ex]
			&& \qquad\geq\ \big(w_+(\xi)-w_+(\eta)\big)^{p-1}\big(w_+(\xi)\p^p(\xi)-w_+(\eta)\p^p(\eta)\big).
		\end{eqnarray*}
		For the second term on the right-hand side in \eqref{cacc e1} we have
		\begin{align*}
			\big|u(\xi)-u(\eta)\big|^{p-2}\big(u(\xi)-u(\eta)\big)w_+(\xi) & \geq -\big(u(\eta)-u(\xi)\big)_+^{p-1}\big(u(\xi)-k\big)_+\\*[1ex]
			& \geq -\big(u(\eta)-k\big)_+^{p-1}\big(u(\xi)-k\big)_+\\*[1ex]
			& = -w_+^{p-1}(\eta)w_+(\xi)\,,
		\end{align*}
		which yields
		\begin{align}\label{cacc e2}
			\int_{\h^n \smallsetminus B_r}&\int_{B_r}|\eta^{-1} \circ \xi|_{\h^n}^{-Q-sp}|u(\xi)-u(\eta)|^{p-2}\big(u(\xi)-u(\eta)\big)w_+(\xi)\p^p(\xi) \,{\rm d}\xi \,{\rm d}\eta\notag\\*[0.5ex]
			& \geq - \int_{\h^n \smallsetminus B_r}\int_{B_r} |\eta^{-1} \circ \xi|_{\h^n}^{-Q-sp}w_+^{p-1}(\eta)w_+(\xi)\p^p(\xi)  \,{\rm d}\xi \,{\rm d}\eta\notag\\*[0.5ex]
			& \geq -\int_{B_r}w_+(\xi)\p^p(\xi) \,{\rm d}\xi \left(\sup_{\xi \in \textup{supp }\p}\int_{\h^n \smallsetminus B_r}|\eta^{-1} \circ \xi|_{\h^n}^{-Q-sp}w_+^{p-1}(\eta)  \,{\rm d}\eta\right).
		\end{align}
		From \eqref{cacc e1}-\eqref{cacc e2}, we deduce
		\begin{eqnarray}\label{cacc e3}
			0 &\geq& \int_{B_r}\int_{B_r}|\eta^{-1} \circ \xi|_{\h^n}^{-Q-sp}|w_+(\xi)-w_+(\eta)|^{p-1} \big(w_+(\xi)\p^p(\xi)-w_+(\eta)\p^p(\eta)\big) \, \,{\rm d}\xi {\rm d} \eta \notag\\*
			&&-2\int_{B_r}w_+(\xi)\p^p(\xi) \, \,{\rm d}\xi \left(\sup_{\xi \in \textup{supp }\p}\int_{\h^n \smallsetminus B_r}|\eta^{-1} \circ \xi|_{\h^n}^{-Q-sp}w_+^{p-1}(\eta) \, \,{\rm d}\eta\right)\\*
			&& - \int_{B_r}f(\xi,u)w_+(\xi)\p^p(\xi)\, \,{\rm d}\xi.\notag
		\end{eqnarray}
		Let us consider the first term in \eqref{cacc e3}. In the case when~$w_+(\xi) \geq w_+(\eta)$ and~$\p(\xi) \leq \p(\eta)$, we can use forthcoming Lemma~\ref{lem_gamma} to obtain
		\begin{equation}
			\p^p(\xi)\, \geq\, (1-c_p \e)\p^p(\eta)-(1+c_p\e)\e^{1-p}|\p(\xi)-\p(\eta)|^p, \qquad \e \in (0,1].
		\end{equation}
		Choosing
		$$
		\e := \frac{1}{\max\{1,2c_p\}}\frac{w_+(\xi)-w_+(\eta)}{w_+(\xi)} \in (0,1]
		$$
		we have that
		\begin{eqnarray*}
			\big(w_+(\xi)-w_+(\eta)\big)^{p-1}w_+(\xi)\p^p(\xi) & \geq & \big(w_+(\xi)-w_+(\eta)\big)^{p-1}w_+(\xi)\big(\max\big\{\p(\xi),\,\p(\eta)\big\}\big)^p\\*
			& &-\frac{1}{2}\big(w_+(\xi)-w_+(\eta)\big)^p\big(\max\big\{\p(\xi),\,\p(\eta)\big\}\big)^p\\*
			& &-c\big(\max\big\{w_+(\xi),\,w_+(\eta)\big\}\big)^p\big|\p(\xi)-\p(\eta)\big|^p,
		\end{eqnarray*}
		where $c$ depends only on $p$. We now recall that it has been assumed that~$\p(\xi) \leq \p(\eta)$. On the other hand, if~$w_+(\xi)=w_+(\eta)=0$, or if~$w_+(\xi)\geq w_+(\eta)$ and~$\p(\xi) \geq \p(\eta)$, then the estimate above trivially follows. Hence, we have
		\begin{eqnarray*}
			&& \big(w_+(\xi)-w_+(\eta)\big)^{p-1}\big(w_+(\xi)\p^p(\xi)-w_+(\eta)\p^p(\eta)\big)\\*[0.5ex]
			&& \qquad\qquad \geq \, \big(w_+(\xi)-w_+(\eta)\big)^{p-1}\big(w_+(\xi)(\max\{\p(\xi),\p(\eta)\}\big)^p-w_+(\eta)\p^p(\eta))\\*[0.5ex]
			&& \qquad\qquad \quad \,-\frac{1}{2}\big(w_+(\xi)-w_+(\eta)\big)^p\big(\max\big\{\p(\xi),\,\p(\eta)\big\}\big)^p\\*
			&& \qquad\qquad \quad \,-c\big(\max\big\{w_+(\xi),\,w_+(\eta)\big\}\big)^p\big|\p(\xi)-\p(\eta)\big|^p\\*[0.5ex]
			&& \qquad\qquad \geq \,  \frac{1}{2}\big(w_+(\xi)-w_+(\eta)\big)^p\big(\max\big\{\p(\xi),\,\p(\eta)\big\}\big)^p\\*
			&& \qquad\qquad \quad\, -c\big(\max\big\{w_+(\xi),\,w_+(\eta)\big\}\big)^p\big|\p(\xi)-\p(\eta)\big|^p
		\end{eqnarray*}
		whenever $w_+(\xi) \geq w_+(\eta)$. In the case when the opposite inequality holds, again it just suffices to interchange the roles of $\xi$ and $\eta$. Therefore, we have
		\begin{eqnarray}\label{cacc e4}
			&&\int_{B_r}\int_{B_r}|\eta^{-1} \circ \xi|_{\h^n}^{-Q-sp}|w_+(\xi)-w_+(\eta)|^{p-1}\big(w_+(\xi)\p^p(\xi)-w_+(\eta)\p^p(\eta)\big)  \,{\rm d}\xi {\rm d}\eta\notag\\*[0.5ex]
			&&  \geq\, \frac{1}{2}\int_{B_r}\int_{B_r} \big(w_+(\xi)-w_+(\eta)\big)^p\big(\max\big\{\p(\xi),\,\p(\eta)\big\}\big)^p \frac{ \,{\rm d}\xi \,{\rm d}\eta}{|\eta^{-1} \circ \xi|_{\h^n}^{Q+sp}}\\
			&& \quad  \, -\ c\int_{B_r}\int_{B_r}\big(\max\big\{w_+(\xi),\,w_+(\eta)\big\}\big)^p|\p(\xi)-\p(\eta)|^p  \frac{ \,{\rm d}\xi \,{\rm d}\eta}{|\eta^{-1} \circ \xi|_{\h^n}^{Q+sp}}. \notag
		\end{eqnarray}
		Now, we note that
		\begin{eqnarray*}
			\big|w_+(\xi)\p(\xi)-w_+(\eta)\p(\eta)\big|^p &\leq& 2^{p-1}\big|w_+(\xi)-w_+(\eta)\big|^p\big(\max\big\{\p(\xi),\,\p(\eta)\big\}\big)^p\\*
			&&+\,2^{p-1}\big(\max\big\{w_+(\xi),\,w_+(\eta)\big\}\big)^p|\p(\xi)-\p(\eta)|^p.
		\end{eqnarray*}
		Hence, combining the preceding inequality with~\eqref{cacc e3} and \eqref{cacc e4}, it follows
		\begin{eqnarray}\label{cacc e5}
			0 &\geq& \int_{B_r}\int_{B_r}|\eta^{-1} \circ \xi|_{\h^n}^{-Q-sp}\big|w_+(\xi)\p(\xi)-w_+(\eta)\p(\eta)\big|^p  \,{\rm d}\xi \,{\rm d}\eta\\*
			&&\qquad-c\int_{B_r}\int_{B_r}|\eta^{-1} \circ \xi|_{\h^n}^{-Q-sp}\big(\max\big\{w_+(\xi),\,w_+(\eta)\big\}\big)^p\big|\p(\xi)-\p(\eta)\big|^p  \,{\rm d}\xi \,{\rm d}\eta\notag\\*
			&&\qquad-2\int_{B_r}w_+(\xi)\p^p(\xi) \,{\rm d}\xi \left(\sup_{\xi \in \textup{supp }\p}\int_{\h^n \smallsetminus B_r}|\eta^{-1} \circ \xi|_{\h^n}^{-Q-sp}w_+^{p-1}(\eta)  \,{\rm d}\eta\right)\notag\\*
			&&\qquad - \int_{B_r}f(\xi,u)w_+(\xi)\p^p(\xi)\,{\rm d}\xi\,.\notag
		\end{eqnarray}
		
		Moreover, in the second integral in the right-hand side in the display above, we can suppose that $w_+(\xi) \geq w_+(\eta)$ up to interchanging $\xi$ with $\eta$; recall that~$|\eta^{-1} \circ \xi|_{\h^n}=|\xi^{-1}\circ \eta|_{\h^n}$ is symmetric. The last integral in~\eqref{cacc e5} can be finally estimated thanks to the assumption on~$f$. We have
		\begin{equation}\label{eq_caccio6}
			\int_{B_r}f(\xi,u(\xi))w_+(\xi)\p^p(\xi)\, \,{\rm d}\xi 
			\, \leq\, \|f\|_{L^\infty(B_r)}\int_{B_r}w_+(\xi)\p^p(\xi) \, \,{\rm d}\xi.
		\end{equation}
		The desired estimate in~\eqref{caccioppoli} is thus a plain consequence of the estimates in~\eqref{cacc e5} and~\eqref{eq_caccio6}.
	\end{proof}
	
	In the proof above, we made use of the following small inequality, which will be useful in the next section, as well.
	\begin{lemma}
		\label{lem_gamma}
		Let $p \geq 1 $ and $\e \in (0,1]$. Then
		$$
		|a|^p \leq |b|^p +c_p\e|b|^p + (1 + c_p\e)\e^{1-p}|a-b|^p, \qquad c_p :=(p-1){\Gamma}\!\left(\max\{1,\,p-2\}\right),
		$$
		holds for every $a,b \in \r^m$, $m\geq 1$. Here ${\Gamma}$\! stands for the standard Gamma function.
	\end{lemma}
	The proof is straightforward. It follows via convexity and a standard iteration process. See for instance Lemma~3.1 in~\cite{DKP16}.
	
	%
	%
	
	\vspace{2mm}
	\section{Proof of the fractional Logarithmic Lemma}\label{sec_log}
	In this section we prove our second main tool; that is, the Logarithmic Lemma~\ref{lem_log}.

	\begin{proof}[Proof of Lemma~{\ref{lem_log}}] 

        As pointed out in Remark~\ref{constant_c}  we denote by~$\tilde{c}$ the precise constant which satisfies inequalities
		\begin{equation}
						\label{dim0}
						|\eta^{-1}\circ \xi|_{\h^n} \, \geq \, \frac{1}{\tilde{c}}|\eta|_{\h^n} - |\xi|_{\h^n}, \qquad |\eta \circ \xi|_{\h^n} \, \leq \, \tilde{c}(|\eta|_{\h^n}+|\xi|_{\h^n}),
	    \end{equation}
even in the case when~$\tilde{c} \equiv 1$. 

		Fix such a constant $\tilde{c}$ and  choose $r>0$ such that $B_r \equiv B_r(\xi_0) \subset B_{\frac{R}{2\tilde{c}}} \equiv B_{\frac{R}{2\tilde{c}}}(\xi_0)$.
		Consider a smooth cut-off function~$\phi \in C^\infty_0(B_{3r/2})$  such that
	$$
		0 \leq \phi \leq 1, \qquad \phi \equiv 1 \mbox{ in } B_r \qquad \mbox{and }\qquad  |\nabla_{\h^n} \phi| \leq \textbf{c}r^{-1} \mbox{ in } B_{3r/2}.
	$$
		Now, take the following test function~$\psi$ in Definition~\ref{solution to inhomo pbm},
		$$
		\psi=(u+d)^{1-p}\phi^p.
		$$ 
		We have 
		\begin{eqnarray}
			\label{dim12}
			&&\int_{B_{2r}}f(\xi,u(\xi))\big(u(\xi)+d\big)^{1-p}\p^p(\xi) \, \,{\rm d}\xi\notag\\*[0.5ex]
			& &\qquad  =\, \int_{B_{2r}}\int_{B_{2r}}|\eta^{-1}\circ \xi|_{\h^n}^{-Q-sp}  |u(\xi)-u(\eta)|^{p-2}\big(u(\xi)-u(\eta)\big)\notag \\*
			& &\hspace{3.8cm}\times \left[\frac{\phi^p(\xi)}{(u(\xi)+d)^{p-1}}-\frac{\phi^p(\eta)}{(u(\eta)+d)^{p-1}}\right] \,{\rm d}\xi \,{\rm d}\eta\notag \\*
			& &\qquad \quad +\,2 \int_{\h^n \smallsetminus B_{2r}}\int_{B_{2r}}|\eta^{-1}\circ \xi|_{\h^n}^{-Q-sp}|u(\xi)-u(\eta)|^{p-2}\frac{u(\xi)-u(\eta)}{(u(\xi)+d)^{p-1}}\phi^p(\xi) \,{\rm d}\xi \,{\rm d}\eta\notag \\*[1ex]
			& & \qquad \qquad=: \, I_1 +I_2.
		\end{eqnarray}
		
		Using the very definition of the function~$\phi$ and the assumption on the datum~$f$, we can estimate the left-hand side of~\eqref{dim12} as follows,
		\begin{eqnarray*}
			&& \int_{B_{2r}}f(\xi,u(\xi))\big(u(\xi)+d)^{1-p}\p^p(\xi)  \,{\rm d}\xi \\*[0.5ex]
			&&  \quad = \, \int_{B_{2r}}(f(\xi,u))_+(u(\xi)+d)^{1-p}\p^p(\xi) \,{\rm d}\xi - \int_{B_{2r}}(f(\xi,u))_-(u(\xi)+d)^{1-p}\p^p(\xi) \, \,{\rm d}\xi\\*[0.5ex]
			&& \quad  \geq \, \int_{B_r}(f(\xi,u))_+(u(\xi)+d)^{1-p} \,{\rm d}\xi -\|f\|_{L^\infty(B_{2r})}\int_{B_{2r}}(u(\xi)+d)^{1-p}  \,{\rm d}\xi.
		\end{eqnarray*} 
		
		\vspace{2mm}
		We now focus on the remaining integrals on the right-hand side of~\eqref{dim12}. 
		
		We start with $I_1$. If $u(\xi) > u(\eta)$ then we can apply~Lemma~\ref{lem_gamma} stated at the end of Section~\ref{sec_caccioppoli}, by choosing there
		$$
		a=\phi(\xi), \quad b=\phi(\eta),
		$$
		and
		$$
		\e = \delta \frac{u(\xi)-u(\eta)}{u(\xi)+d} \in (0,1), \qquad \delta \in (0,1);
		$$
		since $u \geq 0$ in $B_{2r} \subset B_R$, we can therefore estimate the integrand in~$I_1$ as follows,
		\begin{eqnarray*}
			&& |\eta^{-1}\circ \xi|_{\h^n}^{-Q-sp}  |u(\xi)-u(\eta)|^{p-2}\big(u(\xi)-u(\eta)\big)
			\left[\frac{\phi^p(\xi)}{(u(\xi)+d)^{p-1}}-\frac{\phi^p(\eta)}{(u(\eta)+d)^{p-1}}\right]\\*[0.5ex] 
			&&\qquad \quad \leq\, |\eta^{-1}\circ \xi|_{\h^n}^{-Q-sp} \left(\frac{u(\xi)-u(\eta)}{u(\xi)+d}\right)^{p-1}\phi^p(\eta)\left[1 +c_p\delta\frac{u(\xi)-u(\eta)}{u(\xi)+d}-\left(\frac{u(\xi)+d}{u(\eta)+d}\right)^{p-1}\right]\\* 
			&&\qquad \qquad \, +c_p\delta^{1-p} |\eta^{-1}\circ \xi|_{\h^n}^{-Q-sp} |\phi(\xi)-\phi(\eta)|^p \notag \\*[1ex]
			&& \qquad \quad =: \, J_1 + J_2
		\end{eqnarray*} 
		
		Now, in order to estimate the contribution~$J_1$ we will follow the strategy in the proof of~Lemma~1.3 in~\cite{DKP16}. We firstly notice that
		\begin{equation}
			\label{dim1}
			J_1 \, \equiv \, |\eta^{-1}\circ \xi|_{\h^n}^{-Q-sp} \left(\frac{u(\xi)-u(\eta)}{u(\xi)+d}\right)^{p}\phi^p(\eta) \left[ \frac{1-\left(\frac{u(\eta)+d}{u(\xi)+d}\right)^{1-p}}{1-\frac{u(\eta)+d}{u(\xi)+d}}    +c_p\delta \right]
		\end{equation}
		Secondly, we can consider the real function~$g$ given by 
		$$
		g(t):= \frac{1-t^{1-p}}{1-t}=-\frac{p-1}{1-t}\int_{t}^{1}\tau^{-p}{\rm d}\tau, \qquad \forall t \in (0,1).
		$$
		Since $g$ is an increasing function, we have
		$$
		g(t) \leq -(p-1) \quad \forall t \in (0,1).
		$$
		Moreover, for any~$t \leq 1/2$, 
		$$
		g(t) \leq -\frac{p-1}{2p}\frac{t^{1-p}}{1-t}.
		$$
		Therefore, in the case when
		$$
		t=\frac{u(\eta)+d}{u(\xi)+d} \in \left(0,\,\frac{1}{2}\right];
		$$
		i.~e., 
		$$
		u(\eta) + d \leq \frac{u(\xi)+d}{2},
		$$
		then, since $\big(u(\xi)-u(\eta)\big)\big(u(\eta)+d\big)^{p-1}/\big(u(\xi)+d\big)^p \leq 1$, we get
		\begin{equation}
			\label{dim2}
			J_1 \, \leq \, |\eta^{-1}\circ \xi|_{\h^n}^{-Q-sp} \left(c_p\delta -\frac{p-1}{2^p}\right)\left[\frac{u(\xi)-u(\eta)}{u(\eta)+d}\right]^{p-1}\phi^p(\eta).
		\end{equation}
		Hence, it suffices to choose the following suitable~$\delta$ in the preceding inequality, 
		\begin{equation*}
			\label{dim3}
			\delta=\frac{p-1}{2^{p+1}c_p},
		\end{equation*}
		to get
		$$
		J_1 \, \leq\, -|\eta^{-1}\circ \xi|_{\h^n}^{-Q-sp} \frac{p-1}{2^{p+1}}\left[\frac{u(\xi)-u(\eta)}{u(\eta)+d}\right]^{p-1}\phi^p(\eta).
		$$
		
		\vspace{2mm}
		We consider case when
		$$
		u(\eta) + d > \frac{u(\xi)+d}{2};
		$$
		i.~\!e.,
		$$
		t=\frac{u(\eta)+d}{u(\xi)+d} \in \left(\frac{1}{2},\,1\right),
		$$
		we can choose the parameter~$\delta$ as in (\ref{dim3}), and we have
		\begin{equation}
			\label{dim4}
			J_1 \leq -|\eta^{-1}\circ \xi|_{\h^n}^{-Q-sp} \frac{(2^{p+1}-1)(p-1)}{2^{p+1}}\left[\frac{u(\xi)-u(\eta)}{u(\xi)+d}\right]^p\phi^p(\eta).
		\end{equation}
		\vspace{1mm}
		
		Furthermore, if $
		2(u(\eta) + d) < u(\xi)+d
		$, then\begin{equation}
			\label{dim5}
			\left[\log\left(\frac{u(\xi)+d}{u(\eta)+d}\right)\right]^p \leq c_p \left[\frac{u(\xi)-u(\eta)}{u(\eta)+d}\right]^{p-1}
		\end{equation}
		holds, where we  used the fact that $(\log x)^p \leq c(x-1)^{p-1}$ when $x >2$. 
		\vspace{2mm}
		
		On the other hand, if  $
		2(u(\eta) + d) \geq u(\xi)+d
		$, then -- recalling that we have assumed $u(\xi) > u(\eta)$ -- we have
		\begin{equation}
			\label{dim6}
			\left[\log\left(\frac{u(\xi)+d}{u(\eta)+d}\right)\right]^p
			\, =\, \left[\log\left(1 + \frac{u(\xi)-u(\eta)}{u(\eta)+d}\right)\right]^p 
			\,\leq\, 2^p \left[\frac{u(\xi)-u(\eta)}{u(\xi)+d}\right]^p,
		\end{equation}
		where we  used
		$$
		\log(1 + x) \leq x, \quad \forall \, x \geq 0, \qquad \mbox{with } x = \frac{u(\xi)-u(\eta)}{u(\eta)+d} \leq \frac{2[u(\xi)-u(\eta)]}{u(\xi)+d}.
		$$
		\vspace{1mm}
		
		Hence, combining (\ref{dim2}) with (\ref{dim4}), (\ref{dim5}), and (\ref{dim6}), we can conclude with
		\begin{eqnarray*}
			&& |\eta^{-1}\circ \xi|_{\h^n}^{-Q-sp}  |u(\xi)-u(\eta)|^{p-2}\big(u(\xi)-u(\eta)\big) \left[\frac{\phi^p(\xi)}{(u(\xi)+d)^{p-1}}-\frac{\phi^p(\eta)}{(u(\eta)+d)^{p-1}}\right]\\*[0.5ex] 
			&& \qquad \qquad \qquad \leq -\frac{1}{c_p}|\eta^{-1}\circ \xi|_{\h^n}^{-Q-sp}\left[\log\left(\frac{u(\xi)+d}{u(\eta)+d}\right)\right]^p\phi^p(\eta) \\*
			&&  \qquad  \qquad \qquad \quad +c_p\, \delta^{1-p} |\eta^{-1}\circ \xi|_{\h^n}^{-Q-ps} |\phi(\xi)-\phi(\eta)|^p.
		\end{eqnarray*}

		Notice that if $u(\xi)=u(\eta)$, then the same estimate above does trivially hold. If $u(\eta)>u(\xi)$ we can interchange the roles of $\xi$ and $\eta$ in the  computation above. Finally, we get the estimate for the integral~$I_1$ in~(\ref{dim1}),
		\begin{eqnarray}
			\label{dim7}
			I_1 &  \leq & -\frac{1}{c_{p,\delta}}\int_{B_{2r}}\int_{B_{2r}}|\eta^{-1}\circ \xi|_{\h^n}^{-Q-sp}\left|\log\left(\frac{u(\xi)+d}{u(\eta)+d}\right)\right|^p\phi^p(\eta) \,{\rm d}\xi \,{\rm d}\eta \\*
			& & +\, c_{p,\delta}\int_{B_{2r}}\int_{B_{2r}}|\eta^{-1}\circ \xi|_{\h^n}^{-Q-sp} |\phi(\xi)-\phi(\eta)|^p \,{\rm d}\xi \,{\rm d}\eta .\notag
		\end{eqnarray}
		
		For the second contribution in~\eqref{dim1}, namely $I_2$, we can proceed as follows. Firstly, we notice that it $\eta \in B_R$, then $u(\eta) \geq0$,  and so
		$$
		\frac{\big(u(\xi)-u(\eta)\big)_+^{p-1}}{\big(u(\xi)+d\big)^{p-1}}\,\leq\, 1 \qquad \forall \xi \in B_{2r}, \eta \in B_R.
		$$
		Moreover, when $\eta \in \h^n \smallsetminus B_R$, 
		$$
		\big(u(\xi)-u(\eta)\big)_+^{p-1} \, \leq\, 2^{p-1}\big[u^{p-1}(\xi)+(u(\eta))_-^{p-1}\big], \qquad \forall \xi \in B_{2r}.
		$$
		Then, since $u(\xi) \geq 0$ and $\phi(\xi) \leq 1$ on $B_{2r}$, the integral~$I_2$ can be estimated as follows,
		\begin{eqnarray*}
			I_2 & \leq & 2\int_{B_R \smallsetminus B_{2r}}\int_{B_{2r}}|\eta^{-1}\circ \xi|_{\h^n}^{-Q-sp}\big(u(\xi)-u(\eta)\big)_+^{p-1}(u(\xi)+d)^{1-p}\phi^p(\xi) \,{\rm d}\xi \,{\rm d}\eta\\*
			&& +2 \int_{\h^n\smallsetminus B_R}\int_{B_{2r}}|\eta^{-1}\circ \xi|_{\h^n}^{-Q-sp}(u(\xi)-u(\eta))_+^{p-1}\big(u(\xi)+d\big)^{1-p}\phi^p(\xi) \,{\rm d}\xi \,{\rm d}\eta \\*[0.5ex]
			& \leq & c_p\int_{B_R \smallsetminus B_{2r}}\int_{B_{2r}}|\eta^{-1}\circ \xi|_{\h^n}^{-Q-sp}\phi^p(\xi) \,{\rm d}\xi \,{\rm d}\eta \\*
			&& + c_p\int_{\h^n \smallsetminus B_R}\int_{B_{2r}}|\eta^{-1}\circ \xi|_{\h^n}^{-Q-sp}\big[u^{p-1}(\xi)+(u(\eta))_-^{p-1}\big]\big(u(\xi)+d\big)^{1-p}\phi^p(\xi)\,{\rm d}\xi \,{\rm d}\eta\,,
		\end{eqnarray*}
		and thus
		\begin{eqnarray}
			\label{dim8}
			I_2 & \leq& c_p\int_{B_R \smallsetminus B_{2r}}\int_{B_{2r}}|\eta^{-1}\circ \xi|_{\h^n}^{-Q-sp}\phi^p(\xi) \,{\rm d}\xi \,{\rm d}\eta\notag\\*
			&& +\, c_pd^{1-p}\int_{\h^n \smallsetminus B_R}\int_{B_{2r}}|\eta^{-1}\circ \xi|_{\h^n}^{-Q-sp}\phi^p(\xi)\,{\rm d}\xi \,{\rm d}\eta\notag \\*
			&& +\,c_pd^{1-p}\int_{\h^n \smallsetminus B_R}\int_{B_{2r}}|\eta^{-1}\circ \xi|_{\h^n}^{-Q-sp}(u(\eta))_-^{p-1}\,{\rm d}\xi \,{\rm d}\eta\notag \\*[1ex]
			&  =:& I_{2,1}+I_{2,2}+I_{2,3}.
		\end{eqnarray}
		
		From now on, in contrast with the proof in the Euclidean case in~\cite{DKP16} where the logarithmic estimates plainly follows, here we need to take care of the Heisenberg framework in order to deal with the tail contribution in~\eqref{loga}.
		Let us estimate the contribution in the right-side of~\eqref{dim8}. The integral~$I_{2,1}$ can be easily estimated by recalling the definition of the cut-off function~$\phi$; we have
		\begin{equation}
			\label{dim14}
			I_{2,1} \, \leq  \,c r^Q\int_{B_R \smallsetminus B_{2r}}\sup_{\xi \in B_{3r/2}}|\eta^{-1}\circ \xi|_{\h^n}^{-Q-sp} \,{\rm d}\eta 
			\,\leq \, \textbf{c}r^{Q-sp},
		\end{equation}
		where $\textbf{c}=\textbf{c}(n,p)$. 
		
		For any $\xi \in B_{2r}$, and any $\eta \in \h^n\smallsetminus B_R$, in view of~\eqref{dim0} and the symmetry of~$|\cdot|_{\h^n}$, we have
		\begin{eqnarray*}
			\frac{|\eta^{-1}\circ \xi_0|_{\h^n}}{|\eta^{-1}\circ \xi |_{\h^n}} & \leq& \frac{\tilde{c}(|\xi^{-1}\circ \xi_0|_{\h^n}+|\eta^{-1}\circ \xi|_{\h^n})}{|\eta^{-1}\circ \xi |_{\h^n}} \\*[0.5ex]
			& =& \tilde{c}+\tilde{c}\frac{|\xi^{-1}\circ \xi_0|_{\h^n}}{|\eta^{-1}\circ \xi |_{\h^n}} \\*[0.5ex]
			& \leq &\tilde{c}+\tilde{c}\frac{2r}{\frac{1}{\tilde{c}}|\eta^{-1}\circ\xi_0|_{\h^n} - | \xi^{-1}\circ \xi_0 |_{\h^n}}  \\*[0.5ex]  
			& \leq &\tilde{c} +
			\tilde{c}\frac{2r}{R/\tilde{c}-2r}\, :=\,C,
		\end{eqnarray*}
		where $C>0$ since $R>2\tilde{c}r$.
		This yields
		\begin{equation}
			\label{dim15}
			I_{2,2}\,   \leq \, \textbf{c}d^{1-p}r^Q\int_{\h^n \smallsetminus B_R}|\eta^{-1}\circ \xi_0|_{\h^n}^{-Q-sp}\,{\rm d}\eta
			\,  \leq \, \textbf{c}d^{1-p}\frac{r^{Q}}{R^{sp}},
		\end{equation}
		where we also used Lemma~\ref{lem1} with $\gamma=sp$ there; the constant $\textbf{c}$ depending only on $n,p$ and $s$.  
		\vspace{2mm}
		
		For what concerns the integral~$I_{2,3}$ in \eqref{dim8}, we have  		\begin{eqnarray}
			\label{dim10}
			I_{2,3} & =&c_pd^{1-p}\int_{\h^n \smallsetminus B_R}\int_{B_{2r}}|\eta^{-1}\circ \xi|_{\h^n}^{-Q-sp}(u(\eta))_-^{p-1}\,{\rm d}\xi \,{\rm d}\eta\notag  \\*[0.5ex]
			& \leq &\textbf{c}d^{1-p}|B_{2r}|\int_{\h^n \smallsetminus B_R}|\eta^{-1}\circ \xi_0|_{\h^n}^{-Q-sp}(u(\eta))_-^{p-1}\,{\rm d}\eta\notag \\*[0.5ex]
			& \leq &\textbf{c} d^{1-p}\frac{r^Q}{R^{sp}}[\textup{Tail}(u_-;\xi_0,R)]^{p-1},
		\end{eqnarray}
		as long as we enlarge the constant~$\textbf{c}=\textbf{c}(n,p,s)$. Combining (\ref{dim8}),(\ref{dim14}),(\ref{dim15}) with (\ref{dim10}) we finally obtain
		\begin{equation}
			\label{dim11}
			I_2 \, \leq \,  \textbf{c}r^{Q-sp} + \textbf{c}d^{1-p}r^{Q}R^{-sp}+ \textbf{c}r^Qd^{1-p}R^{-sp}[\textup{Tail}(u_-;\xi_0,R)]^{p-1}.
		\end{equation}
		Combining, now, (\ref{dim12}) with (\ref{dim7}) and (\ref{dim11}) we have
		\begin{eqnarray}
			\label{dim17}
			&&	\int_{B_r}\int_{B_r}|\eta^{-1}\circ \xi|_{\h^n}^{-Q-sp}\left|\log\left(\frac{u(\xi)+d}{u(\eta)+d}\right)\right|^p\phi^p(\eta)\,{\rm d}\xi \,{\rm d}\eta \notag \\*[0.5ex]
			&&\qquad\qquad\qquad \leq \textbf{c}\int_{B_{2r}}\int_{B_{2r}}|\eta^{-1}\circ \xi|_{\h^n}^{-Q-sp} |\phi(\xi)-\phi(\eta)|^p\,{\rm d}\xi \,{\rm d}\eta \\*
			&&\qquad\qquad\qquad\quad +\, \textbf{c}r^{Q-sp} + \textbf{c}d^{1-p}r^{Q}R^{-sp}+ \textbf{c}r^Qd^{1-p}R^{-sp}[\textup{Tail}(u_-;\xi_0,R)]^{p-1} \nonumber,
		\end{eqnarray}
		where $\textbf{c}$ depends only on $n,p$ and $s$. Let us consider the first integral on the right-hand side in (\ref{dim17}). Fixed $\eta \in B_{2r}$ there exists a vector $\tilde{\xi}$ such that $\xi = \eta \circ \tilde{\xi}$. Then we can rewrite 
		$$
		|\phi(\xi)-\phi(\eta)|=|\phi(\eta \circ \tilde{\xi})-\phi(\eta)|.
		$$
		Thanks to Theorem~1.41 in \cite{FS82} we have that the previous quantity can be estimated as follows,
		$$
		|\phi(\eta \circ \tilde{\xi})-\phi(\eta)| \leq \textbf{c}|\tilde{\xi}|_{\h^n}\sup_{B_{3r/2}}|\nabla_{\h^n} \phi|
		$$

		Hence, changing again variables and recalling the estimate from above on $|\nabla_{\h^n} \phi|$, we get
		\begin{eqnarray}
			\label{dim20}
			&& \int_{B_{2r}}\int_{B_{2r}}|\eta^{-1}\circ \xi|_{\h^n}^{-Q-sp} |\phi(\xi)-\phi(\eta)|^p  \,{\rm d}\xi \,{\rm d}\eta \notag\\*[0.5ex]
			&& \qquad \qquad \leq \, \textbf{c}r^{-p}\int_{B_{2r}}\int_{B_{2r}}|\eta^{-1}\circ \xi|_{\h^n}^{-Q-sp+p}\,{\rm d}\xi \,{\rm d}\eta\,.
		\end{eqnarray}
			For any $\eta \in B_{2r}$ note that a simple application of the triangular inequality yields
			$$
			|\eta^{-1}\circ \xi|_{\mathds{H}^n}\, \leq\, 	|\eta^{-1}\circ \xi_0|_{\mathds{H}^n} + 	|\xi^{-1}\circ \xi_0|_{\mathds{H}^n}\, \le\, 4r, \qquad \forall \xi \in B_{2r},
			$$
		   so that $B_{2r} \subset B_{4r}(\eta)$. Hence, as  seen in the local framework in~\cite{GT11,Tra14}, one can apply Proposition~5.4.4 in \cite{BLU07} to get
         \begin{eqnarray*}
         	\int_{B_{2r}}\int_{B_{2r}}|\eta^{-1}\circ \xi|_{\h^n}^{-Q-sp+p}\,{\rm d}\xi \,{\rm d}\eta &\le&                                     \int_{B_{2r}}\int_{B_{4r}(\eta)}|\eta^{-1}\circ \xi|_{\h^n}^{-Q-sp+p}\,{\rm d}\xi \, {\rm d}\eta\\*
         	&\leq& Q\omega_n\int_{B_{2r}}\int_0^{4r} \rho^{p-sp-1}\,{\rm d}\rho \, {\rm d}\eta\\*
         	&\leq& \frac{\textbf{c}r^{Q+p-sp}}{p-sp}\,,
         \end{eqnarray*}
         where we denote by~$\omega_n:= |B_1(0)|$. 
		The estimate in~\eqref{dim20} thus becomes
		$$
		\int_{B_{2r}}\int_{B_{2r}}|\eta^{-1}\circ \xi|_{\h^n}^{-Q-sp} |\phi(\xi)-\phi(\eta)|^p  \,{\rm d}\xi \,{\rm d}\eta 
		\, \leq\, \textbf{c}r^{Q-sp},
		$$
		and the proof is complete.
	\end{proof}

	\vspace{2mm}
	
	We conclude this section by presenting an important consequence of the Logarithmic Lemma, which will prove extremely useful in Section~\ref{sec_holder}.
	We firstly need to introduce the following standard notation. Let $v$ be in $L^1(S)$ and denote by $|S|$  the Lebesgue measure of the set $S \subset \h^n$ which we assume to be finite and strictly positive. Here and subsequently we write
	$$
	(v)_S := \,\dashint_{S}v(\xi)\,{\rm d}\xi \, =\, \frac{1}{|S|}\int_S v(\xi)\,{\rm d}\xi.
	$$
	
	We have the following
	\begin{corol}
		\label{s2 corol}
		Let $s \in (0,1)$, $p \in (1,\infty)$, and let $u \in W^{s,p}(\h^n)$ be the solution to problem~\textup{(\ref{problema})} such that $u\geq 0 $ in $B_R \equiv B_R(\xi_0) \subset \Omega$. Let $a,d >0,b>1$ and define
		$$
		v:=\min\big\{(\log(a+d)-\log(u+d))_+,\,\log(b)\big\}.
		$$
		Then, the following estimates holds true, for any $B_r \equiv B_r(\xi_0) \subset B_{R/2\tilde{c}}(\xi_0) \subset B_{R/2}(\xi_0)$ (where $\tilde{c} \geq 1$ is precisely the constant obtained in the proof of the Logarithmic Lemma),
		\begin{eqnarray}\label{eq_corlog}	
			\,\dashint_{B_r}|v-(v)_{B_r}|^p\,{\rm d}\xi & \leq & \textbf{c}+ \textbf{c}d^{1-p}\left(\frac{r}{R}\right)^{sp}\left\{1+\big[\textup{Tail}(u_-;\xi_0,R)\big]^{p-1}\right\}\\*
			&  & + \textbf{c}r^{sp}\|f\|_{L^\infty(B_{2r})}\ \dashint_{B_{2r}}(u(\xi)+d)^{1-p} \, \,{\rm d}\xi. \notag
		\end{eqnarray}
		where $\textbf{c}$ depends only on $n,p,s$.  
	\end{corol}
	\begin{proof}
		The estimate in~\eqref{eq_corlog} is a plain consequence of the Logarithmic Lemma~\ref{lem_log}. Firstly, we need to apply the fractional Poincar\'e inequality, whose proof can be found in~\cite{Min03} (see in particular on Page 297 there. See also the recent paper~\cite{CMPS16} for further Poincar\'e-type inequalities in the Heisenberg group). We get
		$$
		\,\dashint_{B_r}|v-(v)_{B_r}|^p\,{\rm d}\xi\, \leq\, c r^{sp-Q}\int_{B_r}\int_{B_r}|\eta^{-1}\circ \xi|_{\h^n}^{-Q-sp}|v(\xi)-v(\eta)|^p\,{\rm d}\xi \,{\rm d}\eta,
		$$
		where $c=c(n,p,s)$. 
		
		Now, it is sufficient to observe that $v$ is precisely a truncation of the sum of a constant and $\log(u+d)$. For this, we have
		\begin{eqnarray*}
			&&	\int_{B_r}\int_{B_r}|\eta^{-1}\circ \xi|_{\h^n}^{-Q-sp}|v(\xi)-v(\eta)|^p\,{\rm d}\xi \,{\rm d}\eta \\[0.5ex]
			&&\qquad \qquad\qquad\qquad \leq \int_{B_r}\int_{B_r} |\eta^{-1}\circ \xi|_{\h^n}^{-Q-sp}  \left|\log \left(\frac{u(\eta)+d}{u(\xi)+d}\right)\right|^p \,{\rm d}\xi \,{\rm d}\eta\notag \\
			&&\qquad \qquad\qquad\qquad \quad + \int_{B_r}(f(\xi,u))_+\big(u(\xi)+d\big)^{1-p} \, \,{\rm d}\xi,
		\end{eqnarray*}
		so that the desired estimate plainly follow by applying the estimate in~\eqref{loga}.\end{proof}
	
	%
	%
	
	\vspace{2mm}
	\section{Proof of the local boundedness}\label{sec_bdd}
	The aim of this section is to prove the local boundedness result in~Theorem~\ref{teo_bdd}. In our knowledge, such a result is new even in the case of the pure fractional sublaplacian on the Heisenberg group. Here we are able to prove that, via careful estimates based on the nonlocal tail of the solutions together with the Caccioppoli inequality proven in Section~\ref{sec_caccioppoli},
	one can extend the approach firstly seen  in~\cite{DKP16} for the counterpart in the Euclidean case. 
	\begin{proof}[Proof of Theorem~\ref{teo_bdd}] 
		Before starting, we would need to define a few quantities. For any $j \in \mathds{N}$ and $r>0$ such that $B_r(\xi_0) \equiv B_r \subset \Omega$,
		\begin{equation}
			\label{d3 1}
			r_j = \frac{1}{2}(1+2^{-j})r, \qquad \tilde{r}_j = \frac{r_j+r_{j+1}}{2},
		\end{equation}
		$$
		B_j = B_{r_j}(\xi_0),\qquad \tilde{B}_j =B_{\tilde{r}_j}(\xi_0).
		$$
		Also, 
		\begin{equation}
			\label{d3 2}
			\p_j \in C^\infty_0(\tilde{B}_j), \quad 0 \leq \p_j \leq 1, \quad \p_j \equiv 1 \mbox{ on } B_{j+1} \mbox{ and } |\nabla_{\h^n} \p_j| < 2^{j+3}/r,
		\end{equation}
		$$
		k_j = k+ (1-2^{-j})\tilde{k}, \quad \tilde{k}_j = \frac{k_{j+1}+k_j}{2}, \quad \tilde{k} \in \r^+ \mbox{ and } k \in \r,
		$$
		$$
		\tilde{w}_j = (u- \tilde{k}_j)_+ \qquad w_j = (u-k_j)_+.
		$$
		
		\vspace{2mm}
		We divide the proof into two steps. Firstly, we consider the subcritical case when $sp<Q$.
		Recalling the fractional Sobolev exponent $p^* = \displaystyle\frac{Qp}{Q-sp}$, we have
		$$
		\frac{p}{p^*}\,=\, \frac{p}{\frac{Qp}{Q-sp}}\, =\,\frac{Qp-sp^2}{Qp}\,=\, 1-\frac{sp}{Q}.
		$$
		Consequently,
		$$
		\left(\frac{1}{|B_j|}\right)^\frac{p}{p^*} \, =  \,\frac{1}{|B_j|}|B_j|^\frac{sp}{Q} \,=\, \frac{cr_j^{sp}}{|B_j|}.
		$$
		\vspace{2mm}
		
		We now apply the Sobolev inequality in~Theorem~\ref{sobolev}  to the function~$\tilde{w}_j\p_j$. It yields
		\begin{eqnarray*}
			\label{d3 3}
			&& \left(\,\dashint_{B_j}|\tilde{w}_j(\xi)\p_j(\xi)|^{p^*}\,{\rm d}\xi \right)^{\frac{p}{p^*}}\\*[0.5ex]
			&&\qquad \qquad  \qquad \leq cr_j^{sp}\,\dashint_{B_j}\int_{B_j}|\eta^{-1}\circ \xi|_{\h^n}^{-Q-sp}\,\big|\tilde{w}_j(\xi)\p_j(\xi)-\tilde{w}_j(\eta)\p_j(\eta)\big|^p\,{\rm d}\xi \,{\rm d}\eta,
		\end{eqnarray*}
		which, combined with the nonlocal Caccioppoli inequality~\eqref{caccioppoli}, gives
		\begin{eqnarray}\label{estimate}
			&&
			\left(\,\dashint_{B_j}|\tilde{w}_j(\xi)\p_j(\xi)|^{p^*}\,{\rm d}\xi \right)^{\frac{p}{p^*}} \notag\\*[0.5ex]
			&&\qquad\qquad \leq cr_j^{sp}\,\dashint_{B_j}\int_{B_j}|\eta^{-1}\circ \xi|_{\h^n}^{-Q-sp}\tilde{w}^p_j(\xi)|\p_j(\xi)-\p_j(\eta)|^p\,{\rm d}\xi \,{\rm d}\eta\\*
			&&\qquad\qquad \quad+\,cr_j^{sp}\,\dashint_{B_j}\tilde{w}_j(\eta)\p_j^p(\eta)\,{\rm d}\eta\left(\sup_{\eta \in \textup{spt } \p_j}\int_{\h^n \smallsetminus B_j}|\eta^{-1}\circ \xi|_{\h^n}^{-Q-sp}\tilde{w}_j^{p-1}(\xi)\,{\rm d}\xi\right)\notag\\*
			&&\qquad\qquad \quad +\,cr_j^{sp}\|f\|_{L^\infty(B_j)}\,\dashint_{B_j}\tilde{w}_j(\eta)\p_j^p(\eta)\,{\rm d}\eta\notag\\*[1ex]
			&&\qquad\qquad =:\, I_1 + I_2+I_3\,.\notag
		\end{eqnarray}
		
		We begin by estimating the first contribution in the right-hand side of~\eqref{estimate}. By applying the same strategy as in the proof of the Logarithmic Lemma in Section~\ref{sec_log}, we have
		\begin{eqnarray}
			I_1 & \leq & \textbf{c}2^{pj}r_j^{sp-p}\,\dashint_{B_j}\int_{B_j}|\eta^{-1}\circ \xi|_{\h^n}^{-Q-sp+p}\tilde{w}^p_j(\xi)\,{\rm d}\xi \,{\rm d}\eta\notag \\*[0.5ex]
			& \leq & \textbf{c}2^{pj}r_j^{sp-p}\,\dashint_{B_j}\tilde{w}_j^p(\xi)\int_{B_j}|\eta^{-1}\circ \xi|_{\h^n}^{-Q-sp+p}\,{\rm d}\eta \,{\rm d}\xi\notag  \\*[0.5ex]
			& \leq &\textbf{c}2^{pj}\,\dashint_{B_j}\tilde{w}_j^p(\xi)\,{\rm d}\xi.
		\end{eqnarray} 
		\vspace{2mm}
		
		Now, we note that $\tilde{w}_j \leq w^p_j /(\tilde{k}_j-k_j)^{p-1}$. Moreover, since~$\eta \in \textup{spt } \p_j \subseteq \tilde{B}_j$ and~$\xi \in \h^n \smallsetminus B_j$, we have
		$$
		\frac{|\xi^{-1}\circ \xi_0|_{\h^n}}{|\eta^{-1}\circ \xi|_{\h^n}} 
		\,\leq\, \frac{c(|\eta^{-1}\circ \xi_0|_{\h^n}+|\eta^{-1}\circ \xi|_{\h^n})}{|\eta^{-1}\circ \xi|_{\h^n}}
		\,  \leq \, \tilde{c} + \frac{\tilde{c}\tilde{r}_j}{r_j-\tilde{r}_j}
		\, \leq \, \textbf{c}2^{j+4}.
		$$
		For this,
		\begin{eqnarray}
			I_2 & \leq & \textbf{c}2^{j(Q+sp)}r_j^{sp}\,\dashint_{B_j}\frac{w_j^p(\eta)}{(\tilde{k}_j-k_j)^{p-1}}\,{\rm d}\eta\int_{\h^n \smallsetminus B_j}\frac{w^{p-1}_j(\xi)}{|\xi^{-1}\circ \xi_0|_{\h^n}^{Q+sp}}\,{\rm d}\xi\notag\\*[0.5ex]
			& \leq & \frac{\textbf{c}2^{j(Q+sp+p-1)}}{\tilde{k}^{p-1}} [\textup{Tail}(w_0;\xi_0,r/2)]^{p-1}\,\dashint_{B_j}w_j^p(\eta)\,{\rm d}\eta.
		\end{eqnarray}
		
		Since $\tilde{w}_j \leq w_j$, $r_j \leq r$ and $\p \leq 1$, the third integral in~\eqref{estimate} can be easily estimated as follows
		\begin{equation}
			I_3\, \leq \, r^{sp}\|f\|_{L^\infty(B_{r})}\,\dashint_{B_j}w_j^p(\eta)\,{\rm d}\eta.
		\end{equation}
		
		For what concerns the left-hand side in~\eqref{estimate}, we have
		\begin{align*}
			\left(\,\dashint_{B_j}|\tilde{w}_j(\xi)\p_j(\xi)|^{p^*}\,{\rm d}\xi \right)^{\frac{p}{p^*}} & \geq (k_{j+1}-\tilde{k}_j)^{\frac{(p^*-p)p}{p^*}}\left(\,\dashint_{B_{j+1}}w^p_{j+1}(\xi)\,{\rm d}\xi\right)^\frac{p}{p^*}\\*[0.5ex]
			& = \left(\frac{\tilde{k}}{2^{j+2}}\right)^{\frac{(p^*-p)p}{p^*}}\left(\,\dashint_{B_{j+1}}w^p_{j+1}(\xi)\,{\rm d}\xi\right)^\frac{p}{p^*}.
		\end{align*}
		\vspace{2mm}
		
		Then, calling $A_j := \displaystyle\left(\,\dashint_{B_{j}}w^p_{j}(\xi)\,{\rm d}\xi\right)^\frac{1}{p}$ we obtain
		\begin{eqnarray}
			\label{d3 4}
			\left(\frac{\tilde{k}^{1-\frac{p}{p^*}}}{2^{(j+2)\frac{p^*-p}{p^*}}}\right)^p A_{j+1}^\frac{p^2}{p^*}
			&\leq& \textbf{c}2^{j(Q+sp+p-1)}\bigg(1+r^{sp}\|f\|_{L^\infty(B_r)}\\*
   &&+\frac{(\textup{Tail}(w_0;\xi_0,r/2))^{p-1}}{\tilde{k}^{p-1}}\bigg)A_j^p.\nonumber
		\end{eqnarray} 
		\vspace{2mm}
		
		Now, by taking 
		\begin{equation}
			\label{d3 5}
			\tilde{k} \geq \delta [\textup{Tail}(w_0;\xi_0,r/2)], \qquad \delta\in (0,1]\,,
		\end{equation}
		we obtain
		\begin{equation}
			\label{d3 6}
			\left(\frac{A_{j+1}}{\tilde{k}}\right)^\frac{p}{p^*} \leq \delta^\frac{1-p}{p}\bar{c}^\frac{p}{p^*}2^{j\left(\frac{sp}{Q}+\frac{Q+sp+p-1}{p}\right)}\frac{A_j}{\tilde{k}},
		\end{equation}
		where $\bar{c}=\textbf{c}^\frac{p^*}{p^2}[\delta^{p-1}(1+r^{sp}\|f\|_{L^\infty(B_r)})+1]^\frac{p^*}{p^2}2^\frac{2(p^*-p)}{p}$.
		\vspace{2mm}
		
		We set
		$$
		C:= 2^{\frac{sp}{Q-sp}+\frac{Q(Q+sp+p-1)}{p(Q-sp)}}>1, \qquad \beta = \frac{p^*}{p}-1,
		$$
		so that the estimate in~\eqref{d3 6} can be rewritten as
		\begin{equation}
			\label{d3 7}
			\frac{A_{j+1}}{\tilde{k}} \leq \delta^\frac{p^*(1-p)}{p^2}\bar{c}C^j \left(\frac{A_j}{\tilde{k}}\right)^{1+\beta}\,.
		\end{equation}
		Thus it suffices to prove that the following estimate does hold,
		\begin{equation}
			\label{d3 8}
			\frac{A_0}{\tilde{k}} \leq \delta^\frac{p^*(p-1)}{p^2\beta}\bar{c}^{-\frac{1}{\beta}}C^{-\frac{1}{\beta^2}},
		\end{equation}
		and by a standard iteration argument,  
		 it will follow that $A_j \rightarrow 0$ as $j \rightarrow \infty $. 
		\vspace{2mm}
		
		Since
		$$
		\frac{p^*(p-1)}{p^2\beta}=\frac{(p-1)Q}{sp^2},
		$$
		we then choose
		$$
		\tilde{k} := \delta \textup{Tail}(w_0;\xi_0,r/2) + \delta^{-\frac{(p-1)Q}{sp^2}}H A_0, \qquad H = \bar{c}^\frac{1}{\beta}C^\frac{1}{\beta^2},
		$$
		which is in accordance with (\ref{d3 5}).
		
		It follows
		\begin{eqnarray*}
			\sup_{B_{r/2}}u & \leq& k +\tilde{k}\\*[-1ex]
			& =& k+\delta \textup{Tail}((u-k)_+;\xi_0,r/2) +  \delta^{-\frac{(p-1)Q}{sp^2}}H  \left(\,\dashint_{B_r}(u-k)_+^p\right)^\frac{1}{p}, 
		\end{eqnarray*}
		which gives the desired result by choosing $k=0$. The proof is complete in the case when~$sp<Q$.
		\vspace{2mm}
		
		We now turn to the borderline case, when $sp=Q$, which in the Euclidean case in~\cite{DKP16} is mentioned but not even sketched. We fill this gap here, by providing all the details to investigate such a case. Choose $0<s_1 < s <1$;
		in particular $\tilde{w}_j \p_j -(\tilde{w}_j\p_j)_{B_j} \in W^{s,p}(B_j) \subseteq W^{s_1,p}(B_j)$.
		Clearly, $s_1p < sp = Q$, and we have the right to apply the fractional Sobolev inequality which gives, 
		for any $p <q< p_1^*:=\frac{Qp}{Q-s_1p}$,  
		\begin{equation}\label{d3 9a}
			\|\tilde{w}_j \p_j -(\tilde{w}_j\p_j)_{B_j}\|_{L^q(B_j)}  \leq |B_j|^\frac{p_1^*-q}{q p_1^*}\|\tilde{w}_j \p_j -(\tilde{w}_j\p_j)_{B_j}\|_{L^{p_1^*}(B_j)}.
		\end{equation}
		Thus, we can write
		\begin{eqnarray}\label{d3 9}
			&& \left|\left( \,\dashint_{B_j}|\tilde{w}_j(\xi)\p_j(\xi)|^q \,{\rm d}\xi\right)^\frac{1}{q}-\left|\,\dashint_{B_j} \tilde{w}_j(\xi)\p_j(\xi)\,{\rm d}\xi\right| \right|^p\notag\\*
			&&  \leq c \left(\,\dashint_{B_j}|\tilde{w}_j(\xi)\p_j(\xi)-(\tilde{w}_j\p_j)_{B_j}|^{q}\,{\rm d}\xi \right)^{\frac{p}{q}}\notag\\*[1ex]
			&&  \leq c|B_j|^{\frac{(p_1^*-q)p}{q p_1^*}-\frac{p}{q}}\left(\int_{B_j}|\tilde{w}_j(\xi)\p_j(\xi)-(\tilde{w}_j\p_j)_{B_j}|^{p_1^*}\,{\rm d}\xi \right)^{\frac{p}{p_1^*}}\notag\\*[1ex]
			&&   \leq c \frac{r_j^{s_1p}}{r_j^Q}\int_{B_j}\int_{B_j}|\tilde{w}_j(\xi)\p_j(\xi)-\tilde{w}_j(\eta)\p_j(\eta)|^p\ \frac{ \,{\rm d}\xi \,{\rm d}\eta}{|\eta^{-1} \circ \xi|_{\h^n}^{Q+s_1p}},
		\end{eqnarray}
   where in the last inequality  
		we  also used  
		that
		\begin{equation*}
			|B_j|^{\frac{(p_1^*-q)p}{q p_1^*}-\frac{p}{q}}=|B_j|^{-\frac{p}{p_1^*}}=|B_j|^{\frac{s_1p}{Q}-1}=c\frac{r_j^{s_1p}}{r_j^Q}.
		\end{equation*}
		In addition, we notice that
		\begin{eqnarray*}
			&&\left|\left(\,\dashint_{B_j}|\tilde{w}_j(\xi)\p_j(\xi)|^q \,{\rm d}\xi\right)^\frac{1}{q}-\left|\,\dashint_{B_j} \tilde{w}_j(\xi)\p_j(\xi)\,{\rm d}\xi\right| \right|^p\\*
			&& \qquad\qquad\qquad\qquad\geq \, \left(\,\dashint_{B_j}|\tilde{w}_j(\xi)\p_j(\xi)|^q \,{\rm d}\xi\right)^\frac{p}{q}  -\,\dashint_{B_j}|\tilde{w}_j(\xi)\p_j(\xi)|^p d \xi.
		\end{eqnarray*}
		
		We are now in a position to apply the nonlocal Caccioppoli-type inequality in~\eqref{d3 9}, as done before in the subcritical case. Similarly, we get 		\begin{eqnarray*}
			&&\left(\,\dashint_{B_j}|\tilde{w}_j(\xi)\p_j(\xi)|^q\,{\rm d}\xi \right)^{\frac{p}{q}}  \\*
			&& \quad \qquad  \leq \, \textbf{c}\,2^{j(Q+sp+p-1)}\left( 1+r^{sp}\|f\|_{L^\infty(B_r)}+  \frac{[\textup{Tail}(w_0;\xi_0,r/2)]^{p-1}}{\tilde{k}^{p-1}}\right)\,\dashint_{B_j}w_j^p(\eta)\,{\rm d}\eta.
		\end{eqnarray*} 
		
		Moreover, the term on the left-hand side in the inequality above can be estimated as follows
		\begin{equation}
			\left(\,\dashint_{B_j}|\tilde{w}_j(\xi)\p_j(\xi)|^q\,{\rm d}\xi \right)^{\frac{p}{q}} 
			\,\geq\, \left(\frac{\tilde{k}^{1-\frac{p}{q}}}{2^{(j+2)\frac{(q-p)}{q}}}\right)^p\left(\,\dashint_{B_{j+1}}w_{j+1}^p \, \,{\rm d}\xi \right)^{\frac{p}{q}}.
		\end{equation}
		\vspace{2mm}
		
		We set $A_j := \displaystyle\left(\,\dashint_{B_{j}}w^p_{j}(\xi)\,{\rm d}\xi\right)^\frac{1}{p}$
		and we choose $\tilde{k}$ as in~\eqref{d3 5}. It yields
		\begin{equation}
			\label{d3 10}
			\frac{A_{j+1}}{\tilde{k}} \leq \delta^{\frac{q(1-p)}{p^2}}\bar{c}C^j\left(\frac{A_j}{\tilde{k}}\right)^{1+\beta},
		\end{equation}
		where 
		\begin{eqnarray*}
			&& \bar{c}=\textbf{c}^\frac{q}{p^2}\big[\delta^{p-1}\big(1+r^{sp}\|f\|_{L^\infty(B_r)}\big)+1\big]^\frac{q}{p^2}2^\frac{2(q-p)}{p}, \\*[1ex] 
			&& \displaystyle C:=2^{\frac{q(Q+sp+p-1)}{p^2}+\frac{q-p}{p}}>1, \quad \mbox{and} \quad \beta := \frac{q}{p}-1.
		\end{eqnarray*}
		We can now estimate the term~$A_0$ as in \eqref{d3 8}, by replacing $p^*$ with $q$, and considering as $\bar{c}, C$ and $\beta$ the quantities defined in the display above. Then, 
		$$
		A_j \rightarrow 0 \  \text{as} \ j \rightarrow +\infty.
		$$ 
		Taking 
		$$
		\tilde{k} := \delta \textup{Tail}(w_0;\xi_0,r/2) + \delta^{-\frac{(p-1)q}{p(q-p)}}H A_0, \qquad H = \bar{c}^\frac{1}{\beta}C^\frac{1}{\beta^2},
		$$
		which is in clear accordance with~\eqref{d3 5}, we finally deduce that
		\begin{equation*}
			\sup_{B_{r/2}}u 
			\, \leq\, k+\delta \textup{Tail}((u-k)_+;\xi_0,r/2) + \delta^{-\frac{(p-1)q}{p(q-p)}}H  \left(\,\dashint_{B_r}(u-k)_+^p\right)^\frac{1}{p}, 
		\end{equation*}
		which gives the desired result by choosing $k=0$.
	\end{proof}

	%
	%
	
	\vspace{2mm}
	\section{Proof of the H\"older continuity}\label{sec_holder}
	This last section is devoted to the proof of the H\"older continuity of solutions, namely Theorem~\ref{teo_holder}. An iteration lemma is the keypoint of the proof, and, as before, we would have to handle the nonlocality of the involved operator, together with the geometry of our settings. A certain care is required and in the proof below all the estimates proven in the previous sections will appear.
	\vspace{2mm}
	
	We need to fix some notation. For any $j \in \mathds{N}$, let $0 < r < R/2$ for some $R$ such that $B_R (\xi_0) \subset \Omega$,
	$$
	r_j := \sigma^j \frac{r}{2}, \quad \sigma \in \left(0,\frac{1}{4\tilde{c}}\right], \quad B_j := B_{r_j}(\xi_0)\,,
	$$ 
    where, recalling Remark~\ref{constant_c}, the constant~$\tilde{c}$ is the one given in Proposition~\ref{prop2}. 
    
	Moreover, let us define 
	$$
	\frac{1}{2}\omega(r_0):= \frac{1}{2}\omega(r)= \textup{Tail}(u;\xi_0,r/2)+ \textbf{c}\left(\,\dashint_{B_r}u_+^p\,{\rm d}\xi\right)^\frac{1}{p},
	$$
	with  
	$\textbf{c}$ as in Theorem \ref{teo_bdd} and
	$$
	\omega(r_j):= \left(\frac{r_j}{r_0}\right)^\alpha \omega(r_0) \qquad \mbox{for some } \alpha < \frac{sp}{p-1}.
	$$ In order to prove the Theorem \ref{teo_holder} it suffices to prove the following
	\begin{lemma}
		Under the notation introduced above, let $u \in W^{s,p}(\h^n)$ be a solution to problem~\textup{(\ref{problema})}. Then
		\begin{equation}
			\label{s4 1}
			\mathop{\textup{osc}}\limits_{B_j} \, u \leq \omega(r_j), \qquad \forall j =0,1,2,\dots
		\end{equation}
	\end{lemma}
	\begin{proof}
		We will proceed by induction. For this, note that by the definition of $\omega(r_0)$ and Theorem~\ref{teo_bdd} (with $\delta =1$ there), the estimates (\ref{s4 1}) does trivially hold for $j=0$, since both the functions $(u)_+$ and $(-u)_+$ are weak subsolution.
		
		Now, we make a strong induction assumption and assume that (\ref{s4 1}) is valid for all $i \in \big\{1,\dots,j\big\}$ for some $j \geq 0$, and then prove that it holds also for $j+1$. We have that either 
		\begin{equation}
			\label{s4 2}
			\frac{\big|2B_{j+1} \cap \big\{u \geq \inf_{B_j}u+\omega(r_j)/2\big\}\big|}{\big|2B_{j+1}\big|} \,\geq\, \frac{1}{2},
		\end{equation}
		or 
		\begin{equation}
			\label{s4 3}
			\frac{\big|2B_{j+1} \cap \big\{u \leq \inf_{B_j}u+\omega(r_j)/2\big\}\big|}{\big|2B_{j+1}\big|} \,\geq\, \frac{1}{2},
		\end{equation}
		must hold. If (\ref{s4 2}) holds, then we set $u_j := u -\inf_{B_j}u$; if (\ref{s4 3}) holds, then we set $u_j := \omega(r_j)-(u-\inf_{B_j}u)$. Consequently, in all the cases we have that $u_j \geq 0$ in $B_j$ and the following estimate holds,
		\begin{equation}
			\label{s4 4}
			\frac{\big|2B_{j+1}\cap \{u_j \geq \omega(r_j)/2\}\big|}{\big|2B_{j+1}\big|}\geq \frac{1}{2}.
		\end{equation}
		Moreover, $u_j$ is a weak solution which satisfies 
		\begin{equation}
			\label{s4 5}
			\sup_{B_i}|u_j| \leq 2 \omega(r_i), \qquad \forall i \in \{1,\dots, j\}.
		\end{equation}
		\vspace{2mm}
		
		We now claim that under the induction assumption we have
		\begin{equation}
			\label{s4 6}
			[\textup{Tail}(u_j;\xi_0,r_j)]^{p-1} \,\leq\, \textbf{c}\sigma^{-\alpha(p-1)}[\omega(r_j)]^{p-1},
		\end{equation}
		where $\textbf{c}$ depends only on $n,p,s$ 
		and $\alpha$; in particular it is independent of $\sigma$. Indeed,
		\begin{eqnarray}\label{tail1}
			[\textup{Tail}(u_j;\xi_0,r_j)]^{p-1} & = & r_j^{sp}\sum_{i=1}^{j}\int_{B_{i-1} \smallsetminus B_i} |u_j(\xi)|^{p-1}|\xi_0^{-1}\circ \xi|_{\h^n}^{-Q-sp}\,{\rm d}\xi \notag\\*
			& &+ r_j^{sp}\int_{\h^n \smallsetminus B_0}|u_j(\xi)|^{p-1}|\xi_0^{-1}\circ \xi|_{\h^n}^{-Q-sp}\,{\rm d}\xi \notag\\*[1ex]
			& \leq& r_j^{sp}\sum_{i=1}^{j}[\sup_{B_{i-1}}|u_j|]^{p-1}\int_{\h^n \smallsetminus B_i}|\xi_0^{-1}\circ \xi|_{\h^n}^{-Q-sp}\,{\rm d}\xi \notag\\*
			&& + r_j^{sp}\int_{\h^n \smallsetminus B_0}|u_j(\xi)|^{p-1}|\xi_0^{-1}\circ \xi|_{\h^n}^{-Q-sp}\,{\rm d}\xi \notag\\*[1ex]
			& \leq& \textbf{c}\sum_{i=1}^{j}\left(\frac{r_j}{r_i}\right)^{sp}[\omega(r_{i-1})]^{p-1},
		\end{eqnarray}
		where in the last line we also used~\eqref{s4 5} and the fact that
		\begin{align*}
			\int_{\h^n \smallsetminus B_0}&|u_j (\xi)|^{p-1}|\xi_0^{-1}\circ \xi|_{\h^n}^{-Q-sp}\,{\rm d}\xi\\*[0.5ex]
			& \leq c r_0^{-sp}\sup_{B_0}|u|^{p-1}+cr_0^{-sp}[\omega(r_0)]^{p-1}+c \int_{\h^n \smallsetminus B_0}|u(\xi)|^{p-1}|\xi_0^{-1}\circ \xi|_{\h^n}^{-Q-sp}\,{\rm d}\xi\\*[0.5ex]
			& \leq c r_1^{-sp}[\omega(r_0)]^{p-1}.
		\end{align*} 
		
		Moreover, we have
		\begin{align}\label{tail2}
			\sum_{i=1}^{j}\left(\frac{r_j}{r_i}\right)^{sp}&[\omega(r_{i-1})]^{p-1}  \notag\\*[0.5ex]
			& = [\omega(r_0)]^{p-1}\left(\frac{r_j}{r_0}\right)^{\alpha(p-1)}\sum_{i=1}^{j}\left(\frac{r_{i-1}}{r_i}\right)^{\alpha(p-1)}\left(\frac{r_j}{r_i}\right)^{sp-\alpha(p-1)} \notag\\*[0.5ex]
			& = [\omega(r_j)]^{p-1}\sigma^{-\alpha(p-1)}\sum_{i=0}^{j-1}\sigma^{i(sp-\alpha(p-1))} \notag\\*[0.5ex]
			& \leq [\omega(r_j)]^{p-1}\frac{\sigma^{-\alpha(p-1)}}{1-\sigma^{sp-\alpha(p-1)}} \notag\\*[0.5ex]
			& \leq \frac{4^{sp-\alpha(p-1)}}{\log(4)(sp-\alpha(p-1))}\sigma^{-\alpha(p-1)}[\omega(r_j)]^{p-1}
		\end{align}
		where we  used the fact that~$\sigma \leq \frac{1}{4 \tilde{c}} \leq \frac{1}{4}$ and $\alpha < sp/(p-1)$. Combining~\eqref{tail1} with~\eqref{tail2} yields the desired estimate in~(\ref{s4 6}).
		\vspace{2mm}
		
		Next, let us consider the function~$v$ defined by
		\begin{equation}
			\label{s4 7}
			v := \min \left\{\left(\log \left(\frac{\omega(r_j)/2+d}{u_j+d}\right)\right)_+,\, k \right\}, \qquad k >0.
		\end{equation}
        Since we have chosen $\sigma \leq \frac{1}{4 \tilde{c}}$ we have that $2B_{j+1} \subset B_\frac{r_j}{2 \tilde{c}} \subset B_j$. Indeed,
        $$
        2r_{j+1} = 2 \sigma^{j+1}\frac{r}{2} \leq \frac{1}{2\tilde{c}}\sigma^j \frac{r}{2}= \frac{r_j}{2 \tilde{c}}.
        $$
        Hence, we can apply Corollary~\ref{s2 corol}, with $a \equiv \omega(r_j)/2$ and $b \equiv \exp(k)$ there, obtaining that
		\begin{eqnarray*}	
			\,\dashint_{2B_{j+1}}|v-(v)_{2B_{j+1}}|^p\,{\rm d}\xi 
			&\leq &  \textbf{c}+ \textbf{c}d^{1-p}\left(\frac{r_{j+1}}{r_j}\right)^{sp}\left\{1+[\textup{Tail}(u_j;\xi_0,r_j)]^{p-1}\right\}\\*
			& & +\,\textbf{c}(2r_{j+1})^{sp}\|f\|_{L^\infty(B_r)}\,\dashint_{4B_{j+1}}(u_j(\xi)+d)^{1-p} \, \,{\rm d}\xi\\*[0.5ex]
			& =:& I_1 + I_2.
		\end{eqnarray*}
		
		In view of the estimate in~(\ref{s4 6}), we can estimate the first term in the right-hand side of the display above as follows
		\begin{eqnarray*}
			I_1 & \leq &  \textbf{c} + \textbf{c}\sigma^{sp}d^{1-p}\{ 1+ \sigma^{-\alpha(p-1)}[\omega(r_j)]^{p-1}\} \\*[0.5ex]
			&= & \textbf{c} + \textbf{c}\sigma^{sp}d^{1-p}\{ 1+ \sigma^{\alpha(p-1)(j-1)}[\omega(r_0)]^{p-1}\},
		\end{eqnarray*}
		where $\textbf{c}$ depends only on $n,p,s$ and $\alpha$. Moreover, choosing
		$$
		d = \sigma^{\frac{sp}{p-1}}.
		$$ 
		we arrive at 
		\begin{equation}\label{stima I_1}
			I_1\, \leq\, \textbf{c}+\textbf{c}\sigma^{\alpha(p-1)(j-1)}[\omega(r_0)]^{p-1}\,=:\,C_1\,<\,\infty,
		\end{equation}
		where we used  the fact that $\alpha < sp/(p-1)$, $\sigma \leq \frac{1}{4\tilde{c}} \leq \frac{1}{4}$, $j >1$, and $\omega(r_0)<\infty$. 
		
		Let us estimate the second integral $I_2$. Notice that since $\sigma < 1/4$, we have
		$$
		4r_{j+1} \,=\, 4 \sigma^{j+1}\frac{r}{2} \,\leq\, \sigma^j \frac{r}{2}\,=\,r_j,
		$$
		so that $4B_{j+1} \subseteq B_j$, and consequently 
		$u_j \geq 0$ on $4B_{j+1}$. Then, in view of the choice of~$d$, we also have $(u_j+d)^{1-p} \leq \sigma^{-sp}$.  
		We arrive at
		\begin{equation}\label{stima I_2}
			I_2 \leq \textbf{c}\sigma^{sp(j-1)}(r/4)^{sp}\|f\|_{L^\infty(B_r)}=:C_2 < \infty.
		\end{equation}
		Hence from \eqref{stima I_1} and \eqref{stima I_2} we obtain
		\begin{equation}
			\label{s4 8}
			\,\dashint_{2B_{j+1}}|v-(v)_{2B_{j+1}}|^p\,{\rm d}\xi \leq C.
		\end{equation}
		for a suitable positive constant~$C$. 
		
		Denote by shortness $\tilde{B}=2B_{j+1}$. In accordance with~(\ref{s4 4}) and~(\ref{s4 7}) we can write
		\begin{align*}
			k & = \frac{1}{|\tilde{B} \cap \{u_j \geq \omega(r_j)/2\}|}\int_{\tilde{B} \cap \{u_j \geq \omega(r_j)/2\}}k \,{\rm d}\xi\\*[0.5ex]
			& = \frac{1}{|\tilde{B} \cap \{u_j \geq \omega(r_j)/2\}|}\int_{\tilde{B} \cap \{v=0\}}k \,{\rm d}\xi \\*[0.5ex]
			& \leq \frac{2}{|\tilde{B}|} \int_{\tilde{B}}(k-v)\,{\rm d}\xi= 2[k-(v)_{\tilde{B}}]\,.
		\end{align*}
		We then integrate the inequality above over the set $\tilde{B} \cap \{v=k\}$ to get
		\begin{align*}
			\frac{|\tilde{B}\cap \{v=k\}|}{|\tilde{B}|}k & \leq \frac{2}{|\tilde{B}|}\int_{\tilde{B}\cap \{v = k\}}[k-(v)_{\tilde{B}}]\,{\rm d}\xi\\*[0.5ex]
			& \leq \frac{2}{|\tilde{B}|}\int_{\tilde{B}}|v-(v)_{\tilde{B}}|\,{\rm d}\xi \leq \textbf{c},
		\end{align*}
		where we also used~\eqref{s4 8}. 
		\vspace{2mm}
		
		We now take
		\begin{eqnarray*}
			k &=& \log \left(\frac{\omega(r_j)/2+ \sigma^{\frac{sp}{p-1}}}{2  \sigma^{\frac{sp}{p-1}}\omega(r_j)+ \sigma^{\frac{sp}{p-1}}}\right) \\*[0.5ex]
			&=&\log\left(\frac{\omega(r_j)}{4 \sigma^{\frac{sp}{p-1}}\omega(r_j)+ 2\sigma^{\frac{sp}{p-1}}}+ \frac{ \sigma^{\frac{sp}{p-1}}}{2 \sigma^{\frac{sp}{p-1}}\omega(r_j)+ \sigma^{\frac{sp}{p-1}}}\right) \ \approx\  \log \left(\frac{1}{ \sigma^{\frac{sp}{p-1}}}\right),
		\end{eqnarray*}
		so that 
		$$
		\frac{|\tilde{B}\cap \{v=k\}|}{|\tilde{B}|}k\, \leq\, \textbf{c}
		$$
		gives
		\begin{equation}
			\label{s4 10}
			\frac{|\tilde{B} \cap \{u_j \leq 2 d \omega(r_j)\}|}{|\tilde{B}|} 
			\,\leq\, \frac{\textbf{c}}{k}
			\,\leq\, \frac{c_{\log}}{\log \left(\frac{1}{\sigma}\right)};
		\end{equation}
		the constant $c_{\log}$ depending only on $n,p$ and $s$.
		\vspace{2mm}
		
		We are now in a position to start a suitable iteration to deduce the desired oscillation reduction. First, for any $i=0,1,2,\dots$, we define
		$$
		\rho_i = r_{j+1}+2^{-i}r_{j+1}, \quad \tilde{\rho}_i:= \frac{\rho_{i+1}+\rho_i}{2}, \quad B^i=B_{\rho_i}, \quad \tilde{B}^{i}=B_{\tilde{\rho}_i},
		$$ 
		and corresponding cut-off functions
		$$
		\phi_i \in C^\infty_0(\tilde{B}^i), \quad 0 \leq \phi_i \leq 1, \quad \phi_i \equiv 1 \mbox{ on } B^{i+1} \quad \mbox{ and } |\nabla_{\h^n} \phi_i| \leq \textbf{c}\rho_i^{-1}.
		$$
		Furthermore, set
		$$
		k_i = (1+2^{-i}) d \omega(r_j), \qquad w_i := (k_i-u_j)_+,
		$$
		and
		$$
		A_i = \frac{|B^i \cap \{u_j \leq k_i\}|}{|B^i|}= \frac{|B^i \cap \{w_i >0\}|}{|B^i|}.
		$$
		Hence, the Caccioppoli inequality~\eqref{caccioppoli}  yields
		\begin{eqnarray}
			\label{s4 11}
			&& \int_{B^i}\int_{B^i}|\eta^{-1}\circ \xi|^{-Q-sp}_{\h^n}|w_i(\xi)\phi_i(\xi)-w_i(\eta)\phi_i(\eta)|^p \,{\rm d}\xi \,{\rm d}\eta\notag\\*[0.5ex]
			&&\quad \leq \textbf{c}\int_{B^i}\int_{B^i}|\eta^{-1}\circ \xi|^{-Q-sp}_{\h^n}w^p_i(\xi)|\phi_i(\xi)-\phi(\eta)|^p\,{\rm d}\xi \,{\rm d}\eta\\*
			&&\qquad  + \textbf{c}\int_{B^i}w_i(\xi)\phi^p_i(\xi)\,{\rm d}\xi \left(\sup_{\eta \in \tilde{B}_i}\int_{\h^n\smallsetminus B^i}w_i^{p-1}(\xi)|\eta^{-1}\circ \xi|^{-Q-sp}_{\h^n}\,{\rm d}\xi\, +\, \|f\|_{L^\infty(B_r)}\right)\notag.
		\end{eqnarray}
		\vspace{2mm}
		
		We now focus on the subcritical case when~$sp<Q$. We make use of the fractional Sobolev inequality, in order to estimate the first term on the right-hand side as follows
		\begin{eqnarray}
		& &	\label{s4 12}
			A_{i+1}^\frac{p}{p^*}(k_i-k_{i+1})^p \nonumber\\*
             & = &\frac{1}{|B^{i+1}|^\frac{p}{p^*}}\left(\int_{B^{i+1} \cap \{u_j \leq k_{i+1}\}}(k_i-k_{i+1})^{p^*}\phi_i^{p^*}(\xi)\,{\rm d}\xi\right)^\frac{p}{p^*}\notag\\*[0.5ex]
			& \leq& \frac{1}{|B^{i+1}|^\frac{p}{p^*}}\left(\int_{B^i }w_i^{p^*}(\xi)\phi_i^{p^*}(\xi)\,{\rm d}\xi\right)^\frac{p}{p^*}\notag\\*[0.5ex]
			& \leq& \textbf{c}r^{sp-Q}_{j+1}\int_{B^i}\int_{B^i}|\eta^{-1}\circ \xi|^{-Q-sp}_{\h^n}|w_i(\xi)\phi_i(\xi)-w_i(\eta)\phi_i(\eta)|^p \,{\rm d}\xi \,{\rm d}\eta.
		\end{eqnarray}
		Recalling that $|\nabla_{\h^n} \phi_i|\leq \textbf{c}2^ir^{-1}_{j+1}$, we can treat the first term in the right-hand side in~\eqref{s4 11} as follows
		\begin{eqnarray}
			\label{s4 13}
			&& r_{j+1}^{sp}\int_{B^i}\int_{B^i}|\eta^{-1}\circ \xi|^{-Q-sp}_{\h^n}w^p_i(\xi)|\phi_i(\xi)-\phi(\eta)|^p\,{\rm d}\xi \,{\rm d}\eta\notag\\*[0.5ex]
			&&\hspace{5cm}\qquad \leq \textbf{c}2^{ip} r^{-p}_{j+1}r^p_{j+1}\int_{B^i \cap \{u_j \leq k_i\}}w^p_i(\xi)\,{\rm d}\xi\notag\\*[0.5ex]
			&&\hspace{5cm}\qquad \leq \textbf{c}2^{ip}[d\omega(r_j)]^p|B^i \cap \{u_j \leq k_i\}|.
		\end{eqnarray}
		Moreover, 
		\begin{equation}
			\label{s4 14}
			\int_{B^i}w_i(\xi)\phi^p_i(\xi)\,{\rm d}\xi\, \leq\, \textbf{c}[d\omega(r_j)]|B^i \cap \{u_j \leq k_i\}|,
		\end{equation}
		holds.  
		
		Now, we notice that for any $\eta \in \tilde{B}^i$, we have
		$$
		\frac{|\xi^{-1}\circ \xi_0|_{\h^n}}{|\eta^{-1}\circ \xi|_{\h^n}} \leq \frac{\tilde{c}(|\eta^{-1}\circ \xi_0|_{\h^n}+|\eta^{-1}\circ \xi|_{\h^n})}{|\eta^{-1}\circ \xi|_{\h^n}} \leq \tilde{c} + \frac{\tilde{c}\tilde{\rho}_i}{\rho_i-\tilde{\rho}_i}\leq \textbf{c}2^i
		$$
		for all $\xi \in \h^n \smallsetminus B^i$  and that
		$$
		B_{r_{j+1}} \equiv B_{j+1} \subset B^i\, \Rightarrow\, \h^n \smallsetminus B^i \subset \h^n \smallsetminus B_{j+1}.
		$$
		Then we have 
		\begin{equation}
			\label{s4 15}
			r_{j+1}^{sp}\left(\sup_{\eta \in \tilde{B}^i}\int_{\h^n \smallsetminus B^i}w_i^{p-1}(\xi)|\eta^{-1}\circ \xi|_{\h^n}^{-Q-sp}\,{\rm d}\xi\right)\, \leq\, \textbf{c}2^{i(Q+sp)}[\textup{Tail}(w_i;\xi_0,r_{j+1})]^{p-1}
		\end{equation}
		Recalling (\ref{s4 6}) and the facts that $w_i \leq 2d\omega(r_j)$ in $B_j$ and $w_i \leq |u_j|+2d\omega(r_j)$ in $\h^n$, we further get
		\begin{eqnarray*}
			&&\left[\textup{Tail}(w_i;\xi_0,r_{j+1})\right]^{p-1}\\*[0.5ex]
			&&\qquad \leq \textbf{c}r_{j+1}^{sp}\int_{B_j\smallsetminus B_{j+1}}w_i^{p-1}(\xi)|\xi^{-1}_0\circ \xi|_{\h^n}^{-Q-sp}\,{\rm d}\xi + \textbf{c}\left(\frac{r_{j+1}}{r_j}\right)^{sp}[\textup{Tail}(w_i;\xi_0,r_j)]^{p-1}\\*[0.5ex]
			&&\qquad \leq \textbf{c}d^{p-1}\omega(r_{j})^{p-1}+\textbf{c}\sigma^{sp}[\textup{Tail}(u_j;\xi_0,r_j)]^{p-1}\\*[0.5ex]
			&&\qquad \leq \textbf{c}\left(1+\frac{\sigma^{sp-\alpha(p-1)}}{d^{p-1}}\right)[d\omega(r_j)]^{p-1}\\*[0.5ex]
			&&\qquad \leq \textbf{c} [d\omega(r_j)]^{p-1},
		\end{eqnarray*}
		where $\textbf{c}$ depends only on $\alpha,p$ and $s$. Combining the  estimate above with~(\ref{s4 15}) we get
		\begin{equation}
			\label{s4 16}
			r_{j+1}^{sp}\left(\sup_{\eta \in \tilde{B}^i}\int_{\h^n \smallsetminus B^i}w_i^{p-1}(\xi)|\eta^{-1}\circ \xi|_{\h^n}^{-Q-sp}\,{\rm d}\xi\right)\leq \textbf{c}2^{i(Q+sp)}[d\omega(r_j)]^{p-1}.
		\end{equation}
		Putting together (\ref{s4 11}), (\ref{s4 12}), (\ref{s4 13}), (\ref{s4 14}) and (\ref{s4 16}), we obtain 
		\begin{equation}
			\label{s4 17}
			A_{i+1}^\frac{p}{p^*}(k_i-k_{i+1})^p \,\leq\, \textbf{c}\big(1+r_{j+1}^{sp}\|f\|_{L^\infty(B_r)}\big)2^{i(Q+sp+p)}[d\omega(r_j)]^pA_i,
		\end{equation}
		which yields, recalling that $r_{j+1} < r$
		$$
		A_{i+1} \leq \textbf{c}\big(1+r^{sp}\|f\|_{L^\infty(B_r)}\big)^\frac{p^*}{p}2^{i(Q+(2+s)p)\frac{p^*}{p}}A_i^{1+\beta}
		$$
		with $\beta:= sp/(Q-sp)$ by the definition of $k_i$'s. 
		
		Now, if we will prove the following estimate on $A_0$,
		\begin{equation}
			\label{s4 18}
			A_0 = \frac{|\tilde{B} \cap \{u_j \leq 2d \omega(r_j)\}|}{|\tilde{B}|}
			\, \leq\, \textbf{c}^{-1/\beta}\big(1+r^{sp}\|f\|_{L^\infty(B_{r})}\big)^{-\frac{p^*}{p\beta}}2^{-\frac{(Q+(2+s)p)p^*}{p\beta^2}} =: \nu^*,
		\end{equation}
		then we can deduce that
		$$
		A_i \to 0 \ \mbox{ as }  \ i \to \infty.
		$$
		Indeed, the condition~(\ref{s4 18}) it is guaranteed by (\ref{s4 10}) choosing
		$$
		\sigma = \min\left\{\frac{1}{4\tilde{c}},\,\text{e}^{-\frac{c_{\log}}{\nu^*}}\right\},
		$$
		which depends only on $n,p,s$ and $\alpha$. We have hence shown that 
		$$
		\mathop{\textup{osc}}\limits_{B_{j+1}} \, u \, \leq \, (1-d)\omega(r_j) 
		\,= \, (1-d)\left(\frac{r_j}{r_{j+1}}\right)^{\alpha}\!\omega(r_{j+1}) 
		\,= \, (1-d)\sigma^{-\alpha}\omega(r_{j+1}).
		$$
		Taking $\alpha \in \left(0, \frac{sp}{p-1}\right)$ small enough that
		$$
		\sigma^\alpha \geq 1-d=1-\sigma^\frac{sp}{p-1}
		$$ 
		leads to
		$$
		\mathop{\textup{osc}}\limits_{B_{j+1}} \, u \leq \omega(r_{j+1}),
		$$
		and the proof is complete in the case when~$sp<Q$.
		\vspace{2mm}
		
		For the remaining case, namely when~$sp =Q$, we can proceed as in the proof of the local boundedness. Consider $0<s_1<s<1$; the fractional Sobolev inequality gives
		$$
		\|w_i\p_i\|_{L^{p_1^*}(B^i)}^p \leq c[w_i\p_i]_{s_1,p}^p
		$$
		with $p_1:= \frac{Qp}{Q-s_1p}$.  
		Note also that
		\begin{equation}\label{immersion}
			\|w_i\p_i\|_{L^q(B^i)}  \leq |B^i|^\frac{p_1^*-q}{q p_1^*}\|w_i\p_i\|_{L^{p_1^*}(B^i)},
		\end{equation} 
		and thus,
		\begin{align}\label{s4 19}
			&A_{i+1}^\frac{p}{q}(k_i-k_{i+1})^p\nonumber\\*
        & = \frac{1}{|B^{i+1}|^\frac{p}{q}}\left(\int_{B^{i+1} \cap \{u_j \leq k_{i+1}\}}(k_i-k_{i+1})^{q}\phi_i^{q}(\xi)\,{\rm d}\xi\right)^\frac{p}{q}\notag\\*[0.5ex]
			&\leq \frac{1}{|B^{i+1}|^\frac{p}{q}}\left(\int_{B^i }w_i^{q}(\xi)\phi_i^{q}(\xi)\,{\rm d}\xi\right)^\frac{p}{q}\notag\\*[0.5ex]
			& \leq \frac{|B^i|^\frac{(p_1^*-q)p}{q p_1^*}}{|B^{i+1}|^\frac{p}{q}} \left(\int_{B^i }w_i^{p_1^*}(\xi)\phi_i^{p_1^*}(\xi)\,{\rm d}\xi\right)^\frac{p}{p_1^*}\notag\\*[0.5ex]
			&\leq c \frac{r_{j+1}^{s_1p}}{r_{j+1}^Q}\int_{B^i}\int_{B^i}|\eta^{-1}\circ \xi|_{\h^n}^{-Q-s_1p}|w_i(\xi)\p_i(\xi)-w_i(\eta)\p_i(\eta)|^p \,{\rm d}\xi \,{\rm d}\eta.
		\end{align}
		The term in the right-hand side of~\eqref{s4 19} can be estimated using the nonlocal Caccioppoli inequality as in~\eqref{s4 11} and recalling the estimates \eqref{s4 14} and \eqref{s4 16}. 
		
		All in all, we have
		\begin{equation*}
			A_{i+1}^\frac{p}{q}(k_i-k_{i+1})^p 
			\,\leq\, \textbf{c}(1+r_{j+1}^{sp}\|f\|_{L^\infty(B_r)})2^{i(Q+sp+p)}[d\omega(r_j)]^pA_i,
		\end{equation*}
		which yields 
		\begin{equation}\label{s4 23}
			A_{i+1} \, \leq\, \textbf{c}\big(1+r^{sp}\|f\|_{L^\infty(B_r)}\big)^\frac{q}{p}2^{i\frac{(Q+(2+s)p)q}{p}}A_i^{1+\beta}, 
		\end{equation}
		where we denoted by $\beta := q/p-1>0$. We now choose the following~$\sigma$ in~\eqref{s4 10},
		$$
		\sigma:= \min\left\{\frac{1}{4\tilde{c}},\, {\text{e}}^{-\frac{c_{{\textup{log}}}}{\bar{\nu}}}\right\}
		$$
		with  $\bar{\nu}:= \textbf{c}^{-\frac{1}{\beta}}(1+r^{sp}\|f\|_{L^\infty(B_r)})^{-\frac{q}{p\beta}}2^{-\frac{(Q+(2+s)p)q}{p\beta^2}}$, in order to deduce that
		$$
		A_0 = \frac{|\tilde{B}\cap \{u_j \leq 2d\omega(r_j)\}|}{|\tilde{B}|}
		\, \leq\, \textbf{c}^{-\frac{1}{\beta}}\big(1+r^{sp}\|f\|_{L^\infty(B_r)}\big)^{-\frac{q}{p\beta}}2^{-\frac{(Q+(2+s)p)q}{p\beta^2}}.
		$$
		Hence,
		$$
		A_i \rightarrow 0, \qquad \mbox{as} \qquad i \rightarrow \infty.
		$$
		and the proof is complete, by proceeding exactly as in the case when $sp < Q$.
	\end{proof}
	
	\vspace{4mm}
{\bf	Ethics declarations.}  
\\ On behalf of all authors, the corresponding author states that there is {\bf no} conflict of interest.
	
	%
	%

	\vspace{5mm}

\end{document}